\documentclass[a4paper,11pt]{article}
\usepackage{graphicx,epsfig,picins}
\usepackage{fancyhdr,fancybox}
\usepackage{indentfirst}
\usepackage{titlesec}
\usepackage{verbatim}
\usepackage[sort&compress, numbers]{natbib}
\usepackage{array,dcolumn,tabularx}
\usepackage{setspace}
\usepackage{geometry}
\usepackage{extarrows,chemarrow,xypic} 
\usepackage[small]{caption2}
\usepackage{sectsty}
\usepackage{microtype}
\DisableLigatures[f]{encoding = *, family = * }

\usepackage{times}     

\usepackage{latexsym}
\usepackage{amsmath}   
\usepackage{amssymb}   
\usepackage{amsbsy}
\usepackage{amsthm}
\usepackage{amsfonts}
\usepackage{mathrsfs}  
\usepackage{bm}        
\usepackage{relsize}   

\usepackage{hyperref}  

\newcommand{\paperfont}{\fontsize{11pt}{1.2\baselineskip}\selectfont}
\sectionfont{\fontsize{13}{15}\selectfont}
\begin{document}


\theoremstyle{definition}
\makeatletter
\thm@headfont{\bf}
\makeatother
\newtheorem{theorem}{Theorem}[section]
\newtheorem{definition}[theorem]{Definition}
\newtheorem{lemma}[theorem]{Lemma}
\newtheorem{proposition}[theorem]{Proposition}
\newtheorem{corollary}[theorem]{Corollary}
\newtheorem{remark}[theorem]{Remark}
\newtheorem{example}[theorem]{Example}

\lhead{}
\rhead{}
\lfoot{}
\rfoot{}

\renewcommand{\refname}{References}
\renewcommand{\figurename}{Figure}
\renewcommand{\tablename}{Table}
\renewcommand{\proofname}{Proof}

\title{\textbf{Cycle symmetry, limit theorems, and fluctuation theorems for diffusion processes on the circle}}
\author{Hao Ge$^{1,2}$,\;\;\;Chen Jia$^{3,4}$,\;\;\;Da-Quan Jiang$^{3,5}$ \\
\footnotesize $^1$Biodynamic Optical Imaging Center, Peking University, Beijing 100871, P.R. China \\
\footnotesize $^2$Beijing International Center for Mathematical Research, Peking University, Beijing 100871, P.R. China \\
\footnotesize $^3$LMAM, School of Mathematical Sciences, Peking University, Beijing 100871, P.R. China \\
\footnotesize $^4$Department of Mathematical Sciences, The University of Texas at Dallas, Richardson, Texas 75080, U.S.A.\\
\footnotesize $^5$Center for Statistical Science, Peking University, Beijing 100871, P.R. China \\
\footnotesize Email: haoge@pku.edu.cn (H. Ge), jiac@pku.edu.cn (C. Jia), jiangdq@math.pku.edu.cn (D.Q. Jiang)}
\date{}                              
\maketitle                           
\thispagestyle{empty}                

\paperfont

\begin{abstract}
Cyclic structure and dynamics are of great interest in both the fields of stochastic processes and nonequilibrium statistical physics. In this paper, we find a new symmetry of the Brownian motion named as the quasi-time-reversal invariance. It turns out that such an invariance of the Brownian motion is the key to prove the cycle symmetry for diffusion processes on the circle, which says that the distributions of the forming times of the forward and backward cycles, given that the corresponding cycle is formed earlier than the other, are exactly the same. With the aid of the cycle symmetry, we prove the strong law of large numbers, functional central limit theorem, and large deviation principle for the sample circulations and net circulations of diffusion processes on the circle. The cycle symmetry is further applied to obtain various types of fluctuation theorems for the sample circulations, net circulation, and entropy production rate. 
\end{abstract}

\section{Introduction}
Cyclic Markov processes are widely used to model stochastic systems in physics, chemistry, biology, meteorology, and other disciplines. By a cyclic Markov process, we mean a Markov process whose state space has a cyclic topological structure. In general, we mainly focus on two types of cyclic Markov processes: nearest-neighbor random walks on the discrete circle, also called nearest-neighbor periodic walks, and diffusion processes on the continuous circle.

The path of a cyclic Markov process constantly forms the forward and backward cycles and the cycle dynamics of cyclic Markov processes has been studied for a long time. Among these studies, many authors noticed that nearest-neighbor periodic walks have an interesting symmetry \cite{samuels1975classical, dubins1996gambler, qian2006generalized, ge2008waiting}. Let $T^+$ and $T^-$ be the forming times of the forward and backward cycles by a nearest-neighbor periodic walk, respectively. Samuels \cite{samuels1975classical} and Dubins \cite{dubins1996gambler} proved that although the distributions of $T^+$ and $T^-$ may be different, their distributions, given that the corresponding cycle is formed earlier than its reversed cycle, are the same:
\begin{equation}\label{discretesymmetry}
P(T^+\leq u|T^+<T^-) = P(T^-\leq u|T^-<T^+),\;\;\;\forall u\geq 0.
\end{equation}
This equality, which characterizes the symmetry of the forming times of the forward and backward cycles, is named as the cycle symmetry in this paper. Recently, Jia et al. \cite{jia2016cycle} further generalized the cycle symmetry to general discrete-time and continuous-time Markov chains based on some non-trivial equalities about taboo probabilities.

In the previous studies, the cycle symmetry is mainly discussed for nearest-neighbor periodic walks. Besides, Kendall \cite{kendall1974pole} and Pitman and Yor \cite{pitman1981bessel} have proved the cycle symmetry for Brownian motion with constant drift. It is then natural to ask whether the cycle symmetry also holds for the continuous version of nearest-neighbor periodic walks, namely, diffusion processes on the circle. Intuitively, the answer should be affirmative. A natural idea is to approximate diffusion processes by random walks. However, this idea seems to obscure the essence of the cycle symmetry to some extent. In this paper, we prove the cycle symmetry for diffusion processes on the circle using the intrinsic method of stochastic analysis, instead of approximation techniques. We find that the essence of the cycle symmetry lies in a new symmetry of the Brownian motion named as the quasi-time-reversal invariance. Based on Brownian excursion theory \cite{ito1971poisson} and Williams' Brownian paths decomposition theorem \cite{williams1970decomposing, williams1974path}, we prove that the Brownian motion is invariant under a transformation called quasi-time-reversal (see Definition \ref{reversal2}).

In this paper, we use this new invariance of the Brownian motion to prove the cycle symmetry for diffusion processes on the circle. Let $X$ be a diffusion process on the unit circle. Let $T^+$ and $T^-$ be the forming times of the forward and backward cycles, respectively. Then the cycle symmetry for diffusion processes on the circle can also be formulated as \eqref{discretesymmetry}. Let $T = T^+\wedge T^-$ be the time needed for $X$ to form a cycle for the first time. Then the cycle symmetry can be rewritten as
\begin{equation}
P(T\leq u|T^+<T^-) = P(T\leq u|T^-<T^+),\;\;\;\forall u\geq 0.
\end{equation}
This implies that the time needed for $X$ to form a forward or backward cycle is independent of which one of these two cycles is formed. This independence result is important because it shows that the cycle dynamics for any diffusion process on the circle with any initial distribution is nothing but a renewal random walk (see Theorem \ref{renewal}).

The cycle symmetry established in this paper has some interesting applications. In the literature, the most important functionals associated with a cyclic Markov process are the sample circulations. In fact, the circulation theory of nearest-neighbor random walks and general Markov chains have been well established \cite{jiang2004mathematical, kalpazidou2006cycle}. In this paper, we develop the circulation theory for diffusion processes on the circle. Let $N^+_t$ and $N^-_t$ denote the numbers of the forward and backward cycles formed by the diffusion process $X$ up to time $t$, respectively. Then the sample circulations $J^+_t$ and $J^-_t$ of $X$ along the forward and backward cycles are defined as
\begin{equation}
J^+_t = \frac{1}{t}N^+_t,\;\;\;J^-_t = \frac{1}{t}N^-_t,
\end{equation}
respectively, and the sample net circulation is defined as $J_t = J^+_t-J^-_t$. In this paper, we prove the strong law of large numbers, functional central limit theorems, and large deviation principle for the sample circulations with the aid of the cycle symmetry.

Over the past two decades, the studies on fluctuation theorems (FTs) for Markov processes have become a central topic in nonequilibrium statistical physics \cite{seifert2008stochastic, seifert2012stochastic}. So far, there has been a large amount of literature exploring various types of FTs \cite{jarzynski1997nonequilibrium, kurchan1998fluctuation, lebowitz1999gallavotti, crooks1999entropy, seifert2005entropy, esposito2010three, spinney2012nonequilibrium}. In recent years, the FTs for the sample circulations and net circulation of cyclic Markov processes have been increasingly concerned \cite{qian2006generalized, andrieux2007network, seifert2012stochastic}. Interestingly, the cycle symmetry turns out to be a useful tool to study the FTs for diffusion processes on the circle. In this paper, we prove various types of FTs for the sample circulations and net circulation, including the transient FT, the Kurchan-Lebowitz-Spohn-type FT, the integral FT, and the Gallavotti-Cohen-type FT. These FTs characterize the symmetry of $X$ along the forward and backward cycles from different aspects. In particular, we prove that the sample circulations $(J^+_t,J^-_t)$ satisfy a large deviation principle with rate $t$ and good rate function $I_1$ with the following highly non-obvious symmetry:
\begin{equation}
I_1(x_1,x_2) = I_1(x_2,x_1)-\gamma(x_1-x_2),\;\;\;\forall x_1,x_2\in\mathbb{R},
\end{equation}
where $\gamma$ is a constant characterizing whether $X$ is a symmetric Markov process (see Theorem \ref{reversibility}). Moreover, we prove that the sample net circulation $J_t$ also satisfies a large deviation principle with rate $t$ and good rate function $I_2$ with the following symmetry:
\begin{equation}
I_2(x) = I_2(-x)-\gamma x,\;\;\;\forall x\in\mathbb{R}.
\end{equation}
In nonequilibrium statistical physics, a central concept is entropy production rate. In this paper, we also prove that the sample entropy production rate of diffusion processes on the circle satisfies the Gallavotti-Cohen-type FT.

The structure of this paper is organized as follows. In Section 2, we prove the quasi-time-reversal invariance of the Brownian motion using Brownian excursion theory. In Section 3, we apply such an invariance to prove the cycle symmetry for diffusion processes on the circle. In Section 4, we show that the cycle dynamics of any diffusion process on the circle can be represented as a renewal random walk. Sections 5 is devoted to the limit theorems and large deviations for the sample circulations. Sections 6 is devoted to the applications of the cycle symmetry in nonequilibrium statistical physics, including various types of FTs for the sample circulations, net circulation, and entropy production rate. Section 7 is an appendix including some detailed proofs.

\section{Quasi-time-reversal invariance of Brownian motion}\label{invariance}
In this section, we work with the canonical version of one-dimensional Brownian motion. Let $\mathbb{W}$ denote the Winner space, that is, the set of all continuous paths with values in $\mathbb{R}$, let $P$ denote the Winner measure, and let $\mathscr{F}$ denote the Borel $\sigma$-algebra on $\mathbb{W}$ completed with respect to $P$. Let $W = \{W_t: t\geq 0\}$ be the coordinate process on $\mathbb{W}$ defined as $W_t(w) = w(t)$. Let $\{\mathscr{F}_t:t\geq 0\}$ be the natural filtration of $W$ completed with respect to $P$. Then $W$ is a standard Brownian motion defined on the filtered space $(\mathbb{W},\mathscr{F},\{\mathscr{F}_t\},P)$ satisfying the usual conditions. For each $w\in\mathbb{W}$, let
\begin{equation}\label{tau}
\tau(w) = \inf\{t\geq 0: |w(t)| = 1\}
\end{equation}
be the hitting time of $\{-1,1\}$ by $w$ and let
\begin{equation}\label{gtau}
g_\tau(w) = \sup\{0\leq t<\tau(w): w(t) = 0\}
\end{equation}
be the last zero of $w$ before $\tau(w)$.
\begin{definition}\label{reversal2}
The quasi-time-reversal is a transformation $\phi$ on $\mathbb{W}$ defined as
\begin{equation*}
\phi(w)(t) =
\begin{cases}
w(t), &\textrm{if}\;0\leq t<g_\tau(w),\\
w(g_\tau+\tau-t)-1, &\textrm{if}\;w(\tau)=1\;\textrm{and}\;g_\tau(w)\leq t<\tau(w), \\
w(g_\tau+\tau-t)+1, &\textrm{if}\;w(\tau)=-1\;\textrm{and}\;g_\tau(w)\leq t<\tau(w), \\
w(t)-2, &\textrm{if}\;w(\tau)=1\;\textrm{and}\;t\geq\tau(w), \\
w(t)+2, &\textrm{if}\;w(\tau)=-1\;\textrm{and}\;t\geq\tau(w),
\end{cases}\;\;\;\forall w\in\mathbb{W}.
\end{equation*}
\end{definition}

The definition of the quasi-time-reversal is somewhat complicated at first glance, but its intuitive implication is rather clear. Under the quasi-time-reversal, a continuous path $w$ is taken time reversal between $g_{\tau}(w)$ and $\tau(w)$ and is continuously spliced when needed (see Figure \ref{reversal}).
\begin{figure}[!htb]
\begin{center}
\centerline{\includegraphics[width=0.5\textwidth]{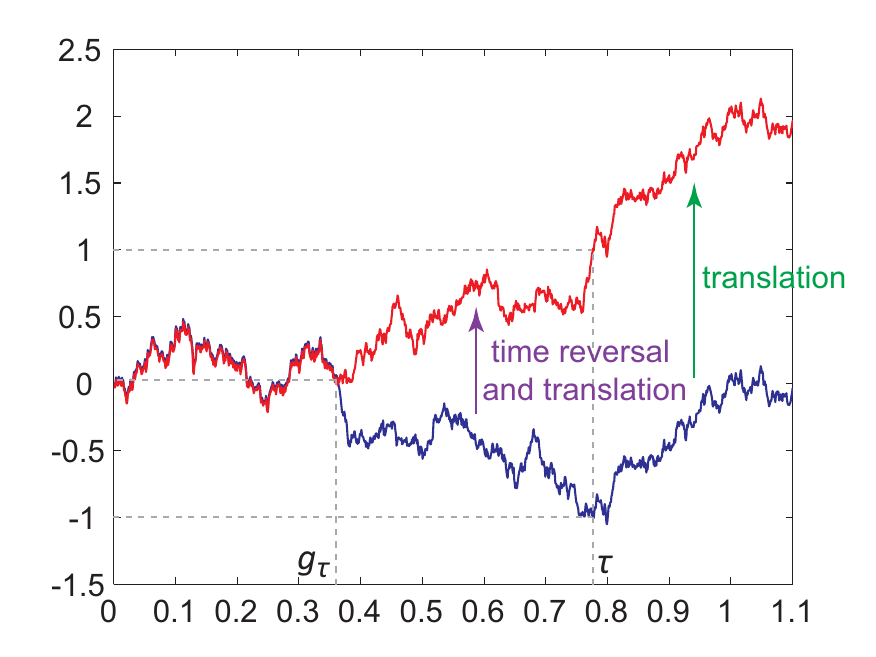}}
\caption{\textbf{A Brownian path and its quasi-time-reversal}. The blue curve represents a Brownian path and the red curve represents the transformed path under the quasi-time-reversal $\phi$. During $[0,g_\tau]$, the two curves coincide. During $[g_\tau,\tau]$, the blue curve is reversed in time and is translated to make the red curve continuous. During $[\tau,\infty)$, the blue curve is translated to make the red curve continuous.}\label{reversal}
\end{center}
\end{figure}

In fact, the expression of the quasi-time-reversal becomes much simpler if it is regarded as a transformation on $\mathbb{W}(S^1)$. Specifically, let $S^1 = \{z\in\mathbb{C}: |z|=1\}$ denote the unit circle in the complex plane. Let $\mathbb{W}(S^1)$ denote the Winner space on $S^1$, that is, the set of all continuous paths with values in $S^1$. Then the map $w\mapsto e^{2\pi iw}$ gives a one-to-one correspondence between $\mathbb{W}$ and $\mathbb{W}(S^1)$. Let $w$ be a path in $\mathbb{W}(S^1)$. Under this one-to-one correspondence, $\tau(w)$ is the time needed for $w$ to form a cycle for the first time and $g_{\tau}(w)$ is the last exit time of $1$ by $w$ before $\tau(w)$. Moreover, the quasi-time-reversal can be viewed as a transformation $\phi$ on $\mathbb{W}(S^1)$ defined as
\begin{equation*}
\phi(w)(t) =
\begin{cases}
w(t), &\textrm{if}\;0\leq t<g_\tau(w),\\
w(g_\tau+\tau-t), &\textrm{if}\;g_\tau(w)\leq t<\tau(w), \\
w(t), &\textrm{if}\;t\geq\tau(w),
\end{cases}\;\;\;\forall w\in\mathbb{W}(S^1).
\end{equation*}

The following theorem shows that the Brownian motion is invariant under the quasi-time-reversal.
\begin{theorem}\label{quasitimereversal}
Let $W$ be a standard one-dimensional Brownian motion. Let $\phi$ be the quasi-time-reversal on $\mathbb{W}$. Then $\phi(W)$ is also a standard one-dimensional Brownian motion.
\end{theorem}

In order to prove the above theorem, we need some knowledge on Brownian excursion theory. We next introduce some notations. For each $w\in\mathbb{W}$, let
\begin{equation*}
R(w) = \inf\{t>0: w(t) = 0\}
\end{equation*}
be the first zero of $w$ after time $0$. Let $\mathbb{U}$ be the set of all continuous paths $w$ such that $0<R(w)<\infty$ and $w(t)=0$ for each $t\geq R(w)$. Let $\mathscr{U}$ be the Borel $\sigma$-algebra on $\mathbb{U}$. In addition, let $n$ be the It\^{o} measure. Then $n$ is a measure on the measurable space $(\mathbb{U},\mathscr{U})$. For the rigorous definition of the It\^{o} measure $n$, please refer to \cite[Chapter XII]{revuz1999continuous}.

For each $w\in\mathbb{W}$, let $i_0(w)$ be the element $u\in\mathbb{U}$ such that
\begin{equation*}
u(t) =
\begin{cases}
w(t), &\textrm{if}\;t<R(w),\\
0, &\textrm{if}\;t\geq R(w).
\end{cases}
\end{equation*}
For each $t\geq 0$, let $i_t(w) = i_0(\theta_t(w))$, where $\theta_t$ is the shift operator defined as $\theta_t(w)(s)=w(t+s)$, and let
\begin{equation*}
g_t(w) = \sup\{0\leq s<t: w(s) = 0\}
\end{equation*}
be the last zero of $w$ before $t$. It is easy to see that $i_{g_t}$ is the excursion straddling the given time $t$.

The next lemma gives the structure of the Brownian motion $W$ between $g_\tau$ and $\tau$. In this paper, a three-dimensional Bessel process starting from $x$ will be abbreviated as BES$^3$($x$).
\begin{lemma}\label{median}
Let $\rho$ and $\tilde\rho$ be two independent BES$^3$(0). Let $\xi$ be a Bernoulli random variable independent of $\rho$ and $\tilde\rho$ with $P(\xi = 0) = P(\xi = 1) = 1/2$. Let $H$ be a process defined as
\begin{equation*}
H_t =
\begin{cases}
\rho_t, &\textrm{if}\;\xi = 0,\\
-\tilde\rho_t, &\textrm{if}\;\xi = 1.
\end{cases}
\end{equation*}
Let $\beta = \inf\{t\geq 0: |H_t|=1\}$. Then the processes $\{W_{g_\tau+t}:0\leq t\leq\tau-g_\tau\}$ and $\{H_t:0\leq t\leq\beta\}$ have the same distribution.
\end{lemma}

\begin{proof}
Let $T$ and $S$ be the hitting times of $1$ and $-1$ by $W$, respectively. Obviously $R$, $S$, $T$, and $\tau$ can be viewed as defined on $\mathbb{U}$. For each $u\in\mathbb{U}$, let $M(u) = \max_{t\leq R(u)}u(t)$. It is a classical result \cite[Chapter XII, Proposition 3.6]{revuz1999continuous} that for any $x>0$,
\begin{equation*}
n(M\geq x) = \frac{1}{2x}.
\end{equation*}
This shows that
\begin{equation}\label{maximum}
n(\tau<R) = n(T<R)+n(S<R) = 2n(T<R) = 2n(M\geq 1) = 1.
\end{equation}
Let $C$ be a nonempty closed set in $\mathbb{R}$ that does not contain 0. Let $T_C$ be the hitting time of $C$ by $W$ and let
\begin{equation*}
g_{T_c}(w) = \sup\{0\leq s<T_c: w(s) = 0\}
\end{equation*}
be the last zero of $w$ before $T_C$. It is a classical result \cite[Chapter XII, Proposition 3.5]{revuz1999continuous} that for any nonnegative measurable function $F$ on $(\mathbb{U},\mathscr{U})$,
\begin{equation}\label{excursion}
E[F(i_{g_{T_C}})|\mathscr{F}_{g_{T_C}}\} = n(F1_{\{T_C<R\}})/n(T_C<R).
\end{equation}
It follows from \eqref{maximum} and \eqref{excursion} that for any $\Gamma\in\mathscr{U}$,
\begin{equation*}
\begin{split}
& P(i_{g_{\tau}}\in\Gamma) = n(\Gamma,\tau<R) = n(\Gamma,T<R)+n(\Gamma,S<R) = n(\Gamma,T<R)+n(-\Gamma,T<R).
\end{split}
\end{equation*}
and that
\begin{equation*}
n(\Gamma,T<R) = n(T<R)P(i_{g_T}\in\Gamma) = \frac{1}{2}P(i_{g_T}\in\Gamma).
\end{equation*}
Thus we have
\begin{equation*}
P(i_{g_{\tau}}\in\Gamma) = \frac{1}{2}[P(i_{g_T}\in\Gamma)+P(i_{g_T}\in-\Gamma)].
\end{equation*}
Let $B$ and $\tilde{B}$ be two independent standard Brownian motions defined on some probability space $(\Omega,\mathscr{F}',P')$ and let $T_1$ and $\tilde{T}_1$ be their hitting times of $1$, respectively. Let $e$ and $\tilde{e}$ be the excursions of $B$ and $\tilde{B}$ straddling $T_1$ and $\tilde{T}_1$, respectively. Let $\eta$ be a Bernoulli random variable independent of $B$ and $\tilde{B}$ with $P'(\eta = 0) = P'(\eta = 1) = 1/2$. Let $J$ be the process defined as
\begin{equation*}
J =
\begin{cases}
e, &\textrm{if}\;\eta = 0,\\
-\tilde{e}, &\textrm{if}\;\eta = 1.
\end{cases}
\end{equation*}
Then we have
\begin{equation*}
P'(J\in\Gamma) = \frac{1}{2}[P'(e\in\Gamma)+P'(\tilde{e}\in-\Gamma)] = \frac{1}{2}[P(i_{g_T}\in\Gamma)+P(i_{g_T}\in-\Gamma)] = P(i_{g_\tau}\in\Gamma).
\end{equation*}
This shows that $i_{g_{\tau}}$ and $J$ have the same distribution. Thus the processes $\{W_{g_\tau+t}: 0\leq t\leq \tau-g_{\tau}\}$ and $\{J_t: 0\leq t\leq\gamma\}$ have the same distribution, where $\gamma$ is the hitting time of $\{-1,1\}$ by $J$. By Williams' Brownian paths decomposition theorem \cite{williams1970decomposing, williams1974path}, the excursion $e$ before reaching 1 is a BES$^{3}$(0) before reaching 1. Thus the processes $\{J_t: 0\leq t\leq\gamma\}$ and $\{H_t:0\leq t\leq\beta\}$ have the same distribution.
\end{proof}

\begin{lemma}\label{independent}
The processes $\{W_t: 0\leq t\leq g_\tau\}$, $\{W_{g_\tau+t}: 0\leq t\leq\tau-g_\tau\}$, and $\{W_{\tau+t}-W_{\tau}: t\geq 0\}$ are independent.
\end{lemma}

\begin{proof}
In view of \eqref{excursion}, it is easy to see that $i_{g_\tau}$ is independent of $\mathscr{F}_{g_\tau}$. This shows that the excursion straddling the hitting time $\tau$ is independent of the past of the Brownian motion up to time $g_\tau$ (see also \cite[Page 492, Lines 1-3]{revuz1999continuous}). Thus the first and second processes are independent. By the strong regenerative property of the Brownian motion, the third process is independent of the first two processes. This completes the proof of this lemma.
\end{proof}

We are now in a position to prove Theorem \ref{quasitimereversal}.
\begin{proof}[Proof of Theorem \ref{quasitimereversal}]
By the definition of $\phi$, the processes $\{W_t: 0\leq t\leq g_\tau\}$ and $\{\phi(W)_t: 0\leq t\leq g_\tau\}$ are the same, and the processes $\{W_{\tau+t}-W_\tau: t\geq 0\}$ and $\{\phi(W)_{\tau+t}-\phi(W)_\tau: t\geq 0\}$ are also the same. By Lemma \ref{independent}, in order to prove that $\phi(W)$ is a Brownian motion, we only need to prove that the processes $\{W_{g_\tau+t}: 0\leq t\leq\tau-g_\tau\}$ and $\{\phi(W)_{g_\tau+t}: 0\leq t\leq\tau-g_\tau\}$ have the same distribution.

We continue to use the notations in Lemma \ref{median}. For any $0\leq t\leq\tau-g_\tau$,
\begin{equation*}
\phi(W)_{g_\tau+t} =
\begin{cases}
W_{\tau-t}-1 &\textrm{if}\; W_\tau = 1, \\
W_{\tau-t}+1 &\textrm{if}\; W_\tau = -1.
\end{cases}
\end{equation*}
Thus the processes $\{\phi(W)_{g_\tau+t}: 0\leq t\leq\tau-g_\tau\}$ and $\{(\rho_{\beta-t}-1)1_{\{\xi=0\}}+(1-\tilde\rho_{\beta-t})1_{\{\xi=1\}}: 0\leq t\leq\beta\}$ have the same distribution. Let $\alpha$ be the hitting time of $1$ by $\rho$. It is a classical result \cite[Chapter VII, Proposition 4.8]{revuz1999continuous} that the processes $\{\rho_t:0\leq t\leq\alpha\}$ and $\{1-\rho_{\alpha-t}:0\leq t\leq\alpha\}$ have the same distribution. This suggests that the processes $\{(\rho_{\beta-t}-1)1_{\{\xi=0\}}+(1-\tilde\rho_{\beta-t})1_{\{\xi=1\}}: 0\leq t\leq\beta\}$ and $\{-\rho_t1_{\{\xi=0\}}+\tilde\rho_t1_{\{\xi=1\}}: 0\leq t\leq\beta\}$ have the same distribution. By Lemma \ref{median}, the processes $\{W_{g_\tau+t}: 0\leq t\leq\tau-g_\tau\}$ and $\{-\rho_t1_{\{\xi=0\}}+\tilde\rho_t1_{\{\xi=1\}}: 0\leq t\leq\beta\}$ have the same distribution. This shows that the processes $\{W_{g_{\tau}+t}: 0\leq t\leq\tau-g_{\tau}\}$ and $\{\phi(W)_{g_{\tau}+t}: 0\leq t\leq\tau-g_{\tau}\}$ have the same distribution.
\end{proof}

\begin{remark}
For any given $x>0$, if we replace $\tau$ by the hitting time of $\{x,-x\}$ and change the definition of the quasi-time-reversal $\phi$ correspondingly, then the result in Theorem \ref{quasitimereversal} still holds.
\end{remark}

\section{Cycle symmetry for diffusion processes on the circle}\label{cyclesymmetry}
In this section, we shall use the quasi-time-reversal invariance of the Brownian motion to prove the cycle symmetry for diffusion processes on the circle.

Let $X = \{X_t: t\geq 0\}$ be a one-dimensional diffusion process with diffusion coefficient $a:\mathbb{R}\rightarrow(0,\infty)$ and drift $b:\mathbb{R}\rightarrow\mathbb{R}$. The initial distribution of $X$ can be arbitrary. We assume that $a$ and $b$ are continuous periodic functions:
\begin{equation*}
a(x+1) = a(x),\;\;\;b(x+1) = b(x),\;\;\;\forall x\in\mathbb{R}.
\end{equation*}
Due to the periodicity, $a$ and $b$ can be viewed as defined on $S^1$ and $X$ can be viewed as a diffusion process on $S^1$. In the sequel, we shall use the symbol $X$ to denote both the diffusion process on $S^1$ and its lifted process on $\mathbb{R}$. The specific implication of $X$ should be clear from the context. We shall construct the diffusion process $X$ as the weak solution to the stochastic differential equation
\begin{equation}\label{SDE}
dX_t = b(X_t)dt+\sigma(X_t)dW_t,
\end{equation}
where $\sigma = a^{1/2}$ and $W$ is a standard Brownian motion defined on some filtered space $(\Omega,\mathscr{G},\mathscr{G}_t,P)$. Since $a$ and $b$ are continuous periodic functions and $a>0$, the Stroock-Varadhan uniqueness theorem \cite[Theorem 7.2.1]{stroock2006multidimensional} ensures that the martingale problem for $(a,b)$ is well-posed. By the equivalence between the martingale-problem formulation and the weak-solution formulation \cite[Theorem 20.1]{rogers2000diffusions2}, the weak solution to \eqref{SDE} always exists and is unique in law.

Recall that the potential function $U$ and scale function $s$ of $X$ are defined as
\begin{equation*}
U(x) = -2\int_0^x\frac{b(y)}{\sigma^2(y)}dy,\;\;\;s(x) = \int_0^xe^{U(y)}dy,
\end{equation*}
respectively. In addition, recall that the distribution $\mu_Y$ of any real-valued continuous process $Y = \{Y_t:t\geq 0\}$ is a probability measure on $\mathbb{W}$ defined as
\begin{equation*}
\mu_Y(A) = P(Y\in A),\;\;\;\forall A\in\mathscr{F}.
\end{equation*}

\begin{definition}\label{formingtimes}
The cycle forming time $T$ of $X$ is defined as
\begin{equation*}
T = \inf\{t\geq 0: |X_t-X_0| = 1\}.
\end{equation*}
If $X_T-X_0 = 1$, then we say that $X$ forms the forward cycle at time $T$. If $X_T-X_0 = -1$, then we say that $X$ forms the backward cycle at time $T$.
\end{definition}

\begin{definition}
For each $n\geq 1$, the $n$th cycle forming time $T_n$ of $X$ is defined as
\begin{equation*}
T_n = \inf\{t\geq T_{n-1}: |X_t-X_{T_{n-1}}| = 1\},
\end{equation*}
where we assume that $T_0 = 0$. If $X_{T_n}-X_{T_{n-1}} = 1$, then we say that $X$ forms the forward cycle at time $T_n$. If $X_{T_n}-X_{T_{n-1}} = -1$, then we say that $X$ forms the backward cycle at time $T_n$.
\end{definition}

Intuitively, $T_n$ is the time needed for $X$ to form a cycle for the $n$th time and $T = T_1$. Under our assumptions, all these cycle forming times $T_n$ are finite a.s. \cite[Chapter VII, Exercise 3.21]{revuz1999continuous}.

\begin{definition}
Let $\nu^+ = \inf\{n\geq 1: \textrm{$X$ forms the forward cycle at time $T_n$}\}$ and let $\nu^- = \inf\{n\geq 1: \textrm{$X$ forms the backward cycle at time $T_n$}\}$. Then the forming time $T^+$ of the forward cycle is defined as $T^+ = T_{\nu^+}$ and the forming time $T^-$ of the backward cycle is defined as $T^- = T_{\nu^-}$.
\end{definition}

Intuitively, $T^+$ is the time needed for $X$ to form a forward cycle for the first time and $T^-$ is the time needed for $X$ to form a backward cycle for the first time. It is clear that $T = T^+\wedge T^-$.

\begin{theorem}\label{general}
Let $X$ be a diffusion process solving the stochastic differential equation \eqref{SDE} with any initial distribution, where $b:\mathbb{R}\rightarrow\mathbb{R}$ and $\sigma:\mathbb{R}\rightarrow(0,\infty)$ are continuous functions with period 1. Then (i) for any $u>0$,
\begin{equation*}
\frac{P(T^+\leq u,T^+<T^-)}{P(T^-\leq u,T^-<T^+)} = \frac{P(T^+<T^-)}{P(T^-<T^+)} = \exp\left(2\int_0^1\frac{b(y)}{\sigma^2(y)}dy\right);
\end{equation*}
(ii) for any $u\geq 0$,
\begin{equation}\label{Haldane}
P(T^+\leq u|T^+<T^-) = P(T^-\leq u|T^-<T^+).
\end{equation}
\end{theorem}

\begin{remark}
The above theorem shows that although the distributions of $T^+$ and $T^-$ may not be the same, their distributions, given that the corresponding cycle is formed earlier than its reversed cycle, are the same. The equality \eqref{Haldane}, which characterizes the symmetry of the forming times of the forward and backward cycles for diffusion processes on the circle, will be named as the cycle symmetry in this paper.
\end{remark}

The proof of the cycle symmetry is divided into three steps. First, using the quasi-time-reversal invariance of the Brownian motion, we shall prove the result when $\sigma\equiv 1$ and $b$ is smooth. Second, using some transformation techniques, we shall prove the result when $\sigma$ and $b$ are both smooth. Third, using some approximation techniques, we shall prove the result when $\sigma$ and $b$ are both continuous.

\begin{lemma}\label{Girsanov}
Let $X$ be a diffusion process solving the stochastic differential equation
\begin{equation}\label{special}
dX_t = b(X_t)dt+dW_t,~~~X_0 = 0,
\end{equation}
where $b:\mathbb{R}\rightarrow\mathbb{R}$ is a $C^1$ function with period 1. Then for any $u\geq 0$ and $A\in\mathscr{F}_u$,
\begin{equation*}
\mu_X(A) = \int_A\exp\left(-\frac{1}{2}\left[U(w_u)+\int_0^u(b^2(w_t)+b'(w_t))dt\right]\right)d\mu_W(w),
\end{equation*}
where $U$ is the potential function of $X$.
\end{lemma}

\begin{proof}
Since $b$ is a continuous periodic function, $b(X_t)$ is a bounded continuous adapted process. By Novikov's condition, the process
\begin{equation*}
M_t = \exp\left(-\int_0^tb(X_s)dW_s - \frac{1}{2}\int_0^tb(X_s)^2ds\right)
\end{equation*}
is a martingale. Let $Q$ be a probability measure defined as $dQ = M_udP$. By Girsanov's theorem, $\{X_t: 0\leq t\leq u\}$ is a standard Brownian motion under $Q$. Thus for any bounded measurable function $f$ on $C[0,u]$,
\begin{equation*}
\begin{split}
& E^Pf(W) = E^Pf(X)\frac{dQ}{dP} = E^Pf(X)\exp\left(\frac{1}{2}\int_0^uU'(X_t)dX_t+\frac{1}{2}\int_0^ub^2(X_t)dt\right).
\end{split}
\end{equation*}
Since $b$ is a $C^1$ function, $U$ is a $C^2$ function. By It\^{o}'s formula, we obtain that
\begin{equation*}
dU(X_t) = U'(X_t)dX_t+\frac{1}{2}U''(X_t)dt = U'(X_t)dX_t-b'(X_t)dt.
\end{equation*}
This shows that
\begin{equation*}
E^Pf(W) = E^Pf(X)\exp\left(\frac{1}{2}\left[U(X_u)+\int_0^u(b^2(X_t)+b'(X_t))dt\right]\right).
\end{equation*}
This implies the result of this lemma.
\end{proof}

We are now in a position to prove the cycle symmetry when $\sigma\equiv1$ and $b$ is smooth.
\begin{lemma}\label{Haldanespecial}
Let $X$ be a diffusion process solving the stochastic differential equation \eqref{special}, where $b:\mathbb{R}\rightarrow\mathbb{R}$ is a $C^1$ function with period 1. Then for any $u>0$,
\begin{equation*}
\frac{P(T\leq u,X_T=1)}{P(T\leq u,X_T=-1)} = \frac{P(X_T=1)}{P(X_T=-1)} = \exp\left(2\int_0^1b(y)dy\right).
\end{equation*}
\end{lemma}

\begin{proof}
Let $\tau$ and $g_\tau$ be defined as in \eqref{tau} and \eqref{gtau}, respectively. For an arbitrarily fixed $u>0$, let $A$ and $B$ be two subsets of $\mathbb{W}$ defined as
\begin{equation}\label{subsets}
A = \{w\in\mathbb{W}: \tau(w)\leq u, w_{\tau}=1\},\;\;\;
B = \{w\in\mathbb{W}: \tau(w)\leq u, w_{\tau}=-1\}.
\end{equation}
Let $\phi$ be the quasi-time-reversal on $\mathbb{W}$. It follows from Theorem \ref{quasitimereversal} that $\phi_*\mu_W = \mu_W$, where $\phi_*\mu_W$ is the push-forward measure defined as $\phi_*\mu_W(\cdot) = \mu_W(\phi^{-1}(\cdot))$. Since $\phi$ is a one-to-one map from $B$ onto $A$, it follows from Lemma \ref{Girsanov} that
\begin{equation*}
\begin{split}
& \mu_X(A) = \int_A\exp\left(-\frac{1}{2}\left[U(w_u)+\int_0^up(w_t)dt\right]\right)d\phi_*\mu_W(w) \\
&= \int_B\exp\left(-\frac{1}{2}\left[U(\phi(w)_u)+\int_0^up(\phi(w)_t)dt\right]\right)d\mu_W(w).
\end{split}
\end{equation*}
where $p = b^2+b'$ is a continuous function with period 1. Since $b$ is a $C^1$ function with period 1, for any $w\in B$,
\begin{equation*}
\begin{split}
& U(\phi(w)_u) = U(w_u+2) = -2\int_0^{w_u}b(y)dy-2\int_{w_u}^{w_u+1}b(y)dy-2\int_{w_u+1}^{w_u+2}b(y)dy \\
&= -2\int_0^{w_u}b(y)dy-4\int_0^1b(y)dy = U(w_u)+2U(1).
\end{split}
\end{equation*}
Since $p$ is a continuous function with period 1, for any $w\in B$,
\begin{equation*}
\begin{split}
& \int_0^up(\phi(w)_t)dt = \int_0^{g_\tau} p(w_t)dt+\int_{g_\tau}^\tau p(w_{g_\tau+\tau-t}+1)dt+\int_\tau^u p(w_t+2)dt \\
&= \int_0^{g_\tau} p(w_t)dt+\int_{g_\tau}^\tau p(w_t)dt+\int_\tau^u p(w_t)dt = \int_0^u p(w_t)dt.
\end{split}
\end{equation*}
The above calculations show that
\begin{equation*}
\mu_X(A) = \int_B\exp\left(-\frac{1}{2}\left[U(w_u)+2U(1)+\int_0^up(w_t)dt\right]\right)d\mu_W(w)
= e^{-U(1)}\mu_X(B).
\end{equation*}
This shows that
\begin{equation*}
\frac{P(T\leq u,X_T=1)}{P(T\leq u,X_T=-1)} = \frac{\mu_X(A)}{\mu_X(B)} = e^{-U(1)}.
\end{equation*}
This completes the proof of this lemma
\end{proof}

We are now in a position to prove the cycle symmetry when $\sigma$ and $b$ are both smooth.
\begin{lemma}\label{smooth}
Let $X$ be a diffusion process solving the stochastic differential equation
\begin{equation}\label{SDEinitial}
dX_t = b(X_t)dt+\sigma(X_t)dW_t,\;\;\;X_0 = 0,
\end{equation}
where $b:\mathbb{R}\rightarrow\mathbb{R}$ and $\sigma:\mathbb{R}\rightarrow(0,\infty)$ are smooth functions with period 1. Then for any $u>0$,
\begin{equation*}
\frac{P(T\leq u,X_T=1)}{P(T\leq u,X_T=-1)} = \frac{P(X_T=1)}{P(X_T=-1)} = \exp\left(2\int_0^1\frac{b(x)}{\sigma^2(x)}dx\right).
\end{equation*}
\end{lemma}

\begin{proof}
Let $f$ be the function on $\mathbb{R}$ defined as
\begin{equation*}
f(x) = \int_0^x\frac{1}{\sigma(y)}dy.
\end{equation*}
It is easy to see that $f$ is a smooth diffeomorphism on $\mathbb{R}$. By It\^{o}'s formula, it is easy to check that $f(X) = \{f(X_t):t\geq 0\}$ is a diffusion process solving the stochastic differential equation
\begin{equation*}
dY_t = c(Y_t)dt + dW_t,\;\;\;Y_0 = 0,
\end{equation*}
where $c = \left(f'b+\frac{1}{2}f''\sigma^2\right)\circ f^{-1}$ is a function satisfying $c(f(\cdot+1)) = c(f(\cdot))$. Since $\sigma$ is a smooth function with period 1, for any $x\in\mathbb{R}$,
\begin{equation*}
f(x+1)-f(x) = \int_x^{x+1}\frac{1}{\sigma(y)}dy = \int_0^1\frac{1}{\sigma(y)}dy = f(1).
\end{equation*}
This shows that $c(\cdot+f(1)) = c(\cdot)$ and thus $c$ is a smooth function with period $f(1)$.

Since $f$ is a smooth diffeomorphism, the cycle forming time $T$ of the diffusion process $X$ is exactly the cycle forming time of the diffusion process $f(X)$:
\begin{equation*}
T = \inf\{t\geq 0: |f(X_t)| = f(1)\}.
\end{equation*}
By Lemma \ref{Haldanespecial}, we have
\begin{equation*}
\frac{P(T\leq u,f(X_T)=f(1))}{P(T\leq u,f(X_T)=-f(1))} = \frac{P(f(X_T)=f(1))}{P(f(X_T)=-f(1))}
= \exp\left(2\int_0^{f(1)}c(y)dy\right).
\end{equation*}
Since $\sigma$ is a smooth function with period 1, it is easy to check that $-f(1) = f(-1)$ and
\begin{equation*}
\int_0^{f(1)}c(y)dy = \int_0^1\left(f'(x)b(x)+\frac{1}{2}f''(x)\sigma^2(x)\right)f'(x)dx
= \int_0^1\frac{b(x)}{\sigma^2(x)}dx.
\end{equation*}
Thus we obtain that
\begin{equation*}
\frac{P(T\leq u,f(X_T)=f(1))}{P(T\leq u,f(X_T)=f(-1))} = \frac{P(f(X_T)=f(1))}{P(f(X_T)=f(-1))}
= \exp\left(2\int_0^1\frac{b(x)}{\sigma^2(x)}dx\right),
\end{equation*}
which gives the desired result.
\end{proof}

\begin{lemma}\label{boundary}
Let $X$ be a diffusion process solving the stochastic differential equation \eqref{SDEinitial}, where $b:\mathbb{R}\rightarrow\mathbb{R}$ and $\sigma:\mathbb{R}\rightarrow(0,\infty)$ are continuous functions with period 1. For an arbitrarily fixed $u>0$, let $A$ and $B$ be two subsets of $\mathbb{W}$ as defined in \eqref{subsets}. Then $\mu_X(\partial{A}) = \mu_X(\partial{B}) = 0$.
\end{lemma}

\begin{proof}
Let $\bar A$ be the closure of $A$ and let $A^\circ$ be the interior of $A$. It is easy to check that
\begin{equation*}
\begin{split}
& \bar{A} \subset A\cup\{\tau<u,\;w(\tau)=-1,\;\exists h>0,\;\textrm{such that}\;\forall t\in[\tau,\tau+h],\;w(t)\geq -1\}, \\
& A^\circ \supset \{\tau<u,\;w(\tau)=1,\;\forall h>0,\;\exists t\in[\tau,\tau+h],\;\textrm{such that}\;w(t)>1\}.
\end{split}
\end{equation*}
The above two relations show that
\begin{equation}\label{threeparts}
\begin{split}
\partial{A} \subset &\;\{\tau=u,w(\tau)=1\} \\
&\cup\{\tau<u,\;w(\tau)=1,\;\exists h>0,\;\textrm{such that}\;\forall t\in[\tau,\tau+h],\;w(t)\leq 1\} \\
&\cup\{\tau<u,\;w(\tau)=-1,\;\exists h>0,\;\textrm{such that}\;\forall t\in[\tau,\tau+h],\;w(t)\geq -1\}.
\end{split}
\end{equation}
Let $M$ be the diffusion process solving the stochastic differential equation
\begin{equation*}
dM_t = \sigma(M_t)dW_t,\;\;\;M_0 = 0.
\end{equation*}
By Girsanov's theorem, $\mu_X$ and $\mu_M$ are equivalent on the measurable space $(\mathbb{W},\mathscr{F}_u)$. Since $M$ is a strong Markov process, it follows from \eqref{threeparts} that
\begin{equation*}
\begin{split}
\mu_M(\partial{A}) \leq &\;P_0(M_u=1)
+P_1(\exists h>0,\;\textrm{such that}\;\forall t\in[0,h],\;M_t\leq 1\} \\
&+P_{-1}(\exists h>0,\;\textrm{such that}\;\forall t\in[0,h],\;M_t\geq -1\},
\end{split}
\end{equation*}
where $P_x(\cdot) = P(\cdot|M_0=x)$. It is a classical result that for each $u>0$ and $x\in\mathbb{R}$, the transition function $P_u(x,dy)$ of $M$ has a density $p(u,x,y)$ with respect to the Lebesgue measure \cite[Lemma 9.2.2]{stroock2006multidimensional}. This shows that $P_0(M_u=1) = 0$ for each $u>0$. Since $M$ is a continuous martingale, there exists a Brownian motion $B$, such that $M_t = B_{[M,M]_t}$ \cite[Chapter V, Theorem 1.6]{revuz1999continuous}, where $[M,M]_t = \int_0^t\sigma^2(M_s)ds$ is a strictly increasing process. Thus
\begin{equation*}
P_1(\exists h>0,\;\textrm{such that}\;\forall t\in[0,h],\;M_t\leq 1\}
= P_1(\exists h>0,\;\textrm{such that}\;\forall t\in[0,h],\;B_t\leq 1\} = 0.
\end{equation*}
Similarly, we can prove that $P_{-1}(\exists h>0,\;\textrm{such that}\;\forall t\in[0,h],\;M_t\geq -1\} = 0$. This shows that $\mu_X(\partial{A})= \mu_M(\partial{A}) = 0$. Similarly, we can prove that $\mu_X(\partial{B}) = 0$.
\end{proof}

The following result can be found in \cite[Theorem 11.1.4]{stroock2006multidimensional}.
\begin{lemma}\label{approximation}
Let $X$ be a diffusion process solving the stochastic differential equation \eqref{SDEinitial}, where $b:\mathbb{R}\rightarrow\mathbb{R}$ and $\sigma:\mathbb{R}\rightarrow(0,\infty)$ are bounded continuous functions. For each $n\geq 1$, let $X_n$ be a diffusion process solving the stochastic differential equation
\begin{equation*}
dX_t = b_n(X_t)dt + \sigma_n(X_t)dW_t,\;\;\;X_0 = 0,
\end{equation*}
where $b_n:\mathbb{R}\rightarrow\mathbb{R}$ and $\sigma_n:\mathbb{R}\rightarrow(0,\infty)$ are bounded continuous functions. Assume that for any $M>0$, the following two conditions hold:
\begin{equation*}
\begin{split}
& \sup_{n\geq 1}\sup_{|x|\leq M}\left(|\sigma_n^2(x)|+|b_n(x)|\right) < \infty, \\
& \lim_{n\rightarrow\infty}\sup_{|x|\leq M}\left(|\sigma_n^2(x)-\sigma^2(x)|+|b_n(x)-b(x)|\right) = 0.
\end{split}
\end{equation*}
Then $\mu_{X_n}\Rightarrow\mu_X$ as $n\rightarrow\infty$, where $\Rightarrow$ stands for weak convergence.
\end{lemma}

We are now in a position to prove the cycle symmetry when $\sigma$ and $b$ are both continuous.
\begin{proof}[Proof of Theorem \ref{general}]
We only need to prove the result when $X$ starts from a given point $x\in {\mathbb R}$. Without loss of generality, we assume that $x = 0$. It is easy to see that (ii) is a direct corollary of (i). Thus we only need to prove (i). Since the subset of smooth periodic functions are dense in the space of continuous periodic functions, we can find a sequence of smooth functions $b_n$ with period 1 such that $b_n$ converges to $b$ uniformly. Similarly, we can find a sequence of smooth functions $\sigma_n>0$ with period 1 such that $\sigma_n$ converges to $\sigma$ uniformly. Let $X_n$ be the diffusion process solving the stochastic differential equation
\begin{equation*}
dX_t = b_n(X_t)dt+\sigma_n(X_t)dW_t,\;\;\;X_0 = 0.
\end{equation*}
For any $u>0$, let $A$ and $B$ be two subsets of $\mathbb{W}$ as defined in \eqref{subsets}. By Lemma \ref{smooth}, we have
\begin{equation}\label{before}
\frac{\mu_{X_n}(A)}{\mu_{X_n}(B)} = \exp\left(2\int_0^1\frac{b_n(x)}{\sigma_n^2(x)}dx\right).
\end{equation}
Since $b,\sigma,b_n,\sigma_n$ are all continuous periodic functions, it is easy to check that the conditions in Lemma \ref{approximation} are satisfied. Thus $\mu_{X_n}\Rightarrow\mu_X$. Moreover, it follows from Lemma \ref{boundary} that $\mu_X(\partial{A}) = \mu_X(\partial{B}) = 0$. This shows that $\mu_{X_n}(A)\rightarrow\mu_X(A)$ and $\mu_{X_n}(B)\rightarrow\mu_X(B)$. Then we can obtain the desired result by letting $n\rightarrow\infty$ in \eqref{before}.
\end{proof}

\section{Renewal random walk representation for diffusion processes on the circle}
In this section, we shall use the cycle symmetry to prove that the cycle dynamics of any diffusion process on the circle with any initial distribution is nothing but a renewal random walk. Here a renewal random walk stands for a nearest-neighbor random walk whose interarrival times are i.i.d. positive random variables.

\begin{proposition}\label{independent2}
$T$ and $X_T-X_0$ are independent.
\end{proposition}

\begin{proof}
It is easy to see that $T^+<T^-$ is equivalent to $X_T-X_0=1$ and $T^-<T^+$ is equivalent to $X_T-X_0=-1$. By Theorem \ref{general}, we obtain that
\begin{equation*}
P(T\leq u|X_T-X_0=1) = P(T\leq u|X_T-X_0=-1).
\end{equation*}
This implies that $T$ and $X_T-X_0$ are independent.
\end{proof}

\begin{remark}\label{independent1}
Intuitively, $T$ is the time needed for $X$ to form a forward or backward cycle for the first time and $X_T-X_0$ characterizes which one of these two cycles is formed. Thus the above corollary shows that the forming time of a forward or backward cycle for a diffusion process on the circle is independent of which one of these two cycles is formed.
\end{remark}

The following result is interesting in its own right.
\begin{theorem}\label{initial}
The distributions of $T$ and $X_T-X_0$ are independent of the initial distribution of $X$.
\end{theorem}

\begin{proof}
Let $\mu$ be the initial distribution of $X$. It follows from Theorem \ref{general} that
\begin{equation*}
\frac{P_\mu(X_T-X_0=1)}{P_\mu(X_T-X_0=-1)} = \exp\left(2\int_0^1\frac{b(y)}{\sigma^2(y)}dy\right).
\end{equation*}
This shows that the distribution of $X_T-X_0$ is independent of the initial distribution of $X$.

We shall now view $X$ as a diffusion process on $S^1$. Without loss of generality, we assume that $X$ is the coordinate process on $W(S^1)$. Let $T_2$ be the second cycle forming time of $X$ and let $S = T_2-T$. By the strong Markov property of $X$, for any $t\geq 0$,
\begin{equation}\label{strong}
P_\mu(S\leq t|\mathscr{G}_T) = P(T\circ\theta_T\leq t|\mathscr{G}_T) = P_{X_T}(T\leq t) = P_\mu(T\leq t),
\end{equation}
where $\theta_T$ be the shift operator on $W(S^1)$ and the last equality holds because $X_T$ and $X_0$ represent the same point on $S^1$. This shows that $T$ and $S$ are independent and have the same distribution. Since \eqref{strong} holds for any initial distribution $\mu$, we have for any $x\in S^1$,
\begin{equation*}
P_\mu(S\leq t,T\leq t) = P_\mu(T\leq t)^2,\;\;\;P_x(S\leq t,T\leq t) = P_x(T\leq t)^2.
\end{equation*}
The above two equations suggest that
\begin{equation*}
\left[\int_{S_1}P_x(T\leq t)\mu(dx)\right]^2 = \int_{S^1}P_x(T\leq t)^2\mu(dx).
\end{equation*}
Since the above equation holds for any distribution $\mu$, the function $g(x) = P_x(T\leq t)$ must be a constant. This shows that the distribution of $T$ is independent of the initial distribution of $X$.
\end{proof}

The above results suggest that the cycle dynamics of any diffusion process on the circle with any initial distribution is simply a renewal random walk: we only need to wait a random time for the process to form a cycle, toss a coin to decide whether the forward or the backward cycle is formed, and repeat the above procedures independently. This fact is stated rigorously in the following theorem.
\begin{theorem}\label{renewal}
For any $t\geq 0$ and $n\geq 1$, let $N_t = \inf\{n\geq 0: T_{n+1}>t\}$ and let $\xi_n = X_{T_n}-X_{T_{n-1}}$, where $T_0 = 0$. Let $L = \{L_t:t\geq 0\}$ be a process defined as $L_t = \sum_{n=1}^{N_t}\xi_n$. Then $L$ is a renewal random walk.
\end{theorem}

\begin{proof}
For any $n\geq 0$, let $\tau_n = T_{n+1}-T_n$. By the strong Markov property of $X$ and Proposition \ref{independent2}, for any $x\in\mathbb{R}$, $t_0,\cdots,t_m\geq 0$, and $s_1,\cdots,s_m\in\{1,-1\}$,
\begin{equation*}
P_x(\tau_0\leq t_0,\cdots,\tau_m\leq t_m,\xi_1=s_1,\cdots,\xi_m=s_m)
= \prod_{i=0}^mP_x(T\leq t_i)\prod_{i=1}^mP_x(X_T-X_0=s_i).
\end{equation*}
By Theorem \ref{initial}, it is easy to see that
\begin{equation*}
P(\tau_0\leq t_0,\cdots,\tau_m\leq t_m,\xi_1=s_1,\cdots,\xi_m=s_m)
= \prod_{i=0}^mP(T\leq t_i)\prod_{i=1}^mP(X_T-X_0=s_i).
\end{equation*}
This shows that $\{\tau_n:n\geq 0\}$ are i.i.d. positive random variables and $\{\xi_n:n\geq 1\}$ are i.i.d. Bernoulli random variables independent of $\{\tau_n:n\geq 0\}$. Thus $L$ is a renewal random walk.
\end{proof}

\section{Limit theorems and large deviations for sample circulations}
In this section, we shall use the renewal random walk representation for diffusion processes on the circle to study the limit theorems and large deviations for the sample circulations.

\subsection{Some results on Markov renewal processes}
To make the paper more self-contained, we recall the following definition.
\begin{definition}
Let $\{\xi_n: n\geq 0\}$ be an irreducible discrete-time Markov chain with finite state space $E$. Assume that each $i\in E$ is associated with a Borel probability measure $\phi_i$ on $(0,\infty)$. Let $\{\tau_n: n\geq 0\}$ be a sequence of positive random variables such that given $\{\xi_n:n\geq 0\}$, the random variables $\{\tau_n:n\geq 0\}$ are independent and have the distribution
\begin{equation*}
P(\tau_n\in\cdot|\xi_1,\cdots,\xi_n) = \phi_{\xi_n}(\cdot).
\end{equation*}
Then $\{(\xi_n,\tau_n): n\geq 0\}$ is called a Markov renewal process.
\end{definition}

\begin{remark}
It is easy to see that a renewal random walk can be represented as a Markov renewal process, where $\{\tau_n:n\geq 0\}$ are i.i.d positive random variables and $\{\xi_n:n\geq 1\}$ are i.i.d. Bernoulli random variables independent of $\{\tau_n:n\geq 0\}$.
\end{remark}

In fact, the limit theorems and large deviations for Markov renewal processes are well-established. Before we state these results, we introduce several notations. Let $\{(\xi_n,\tau_n): n\geq 0\}$ be a Markov renewal process, where $\xi = \{\xi_n: n\geq 0\}$ is an irreducible discrete-time Markov chain on a finite state space $E = \{1,\cdots,N\}$ with transition probability matrix $P = (p_{ij})$ and invariant distribution $\pi = (\pi_1,\cdots,\pi_N)$. For any $n\geq 1$ and $t\geq 0$, let $T_n = \sum_{k=0}^{n-1}\tau_k$ be the $n$th jump time of the Markov renewal process and let $N_t = \inf\{n\geq 0:T_{n+1}>t\}$ be the number of jumps up to time $t$.

Let $r = (r_{ij})$ be a given matrix and define the reward process $R = \{R_t:t\geq 0\}$ by
\begin{equation}\label{flow}
R_t = \sum_{n=1}^{N_t}r_{\xi_{n-1}\xi_n}.
\end{equation}
Let $g = (g_1,\cdots,g_N)'$ be a column vector defined as
\begin{equation*}
g_i = \sum_{j\in E}p_{ij}r_{ij},
\end{equation*}
and let $f$ be a solution to the following Poisson equation:
\begin{equation*}
(I-P)f = g - \pi g1,
\end{equation*}
where $1$ is a column vector whose components are all 1. Since $P$ is irreducible, it is easy to check that the rank of $I-P$ is $N-1$. Thus the solution to the Poisson equation is unique up to an additive constant. For any $i,j\in E$, set
\begin{equation*}
m_i = \int_{[0,\infty)}t\phi_i(dt),\;\;\;
v_i = \int_{[0,\infty)}t^2\phi_i(dt),\;\;\;
H_{ij} = r_{ij}+f_j-f_i.
\end{equation*}
Moreover, set
\begin{equation*}
m = \sum_{i\in E}\pi_im_i,\;\;\;u = m^{-1}\sum_{i,j\in E}\pi_ip_{ij}r_{ij},\;\;\;
\sigma^2 = m^{-1}\sum_{i,j\in E}\pi_ip_{ij}(H_{ij}^2-2um_iH_{ij}+u^2v_i).
\end{equation*}

The following lemma, which gives the strong law of large numbers and functional central limit theorem for Markov renewal processes, is due to Glynn and Hass \cite{glynn2004functional}.
\begin{lemma}\label{FCLT}
Assume that $\sum_{i\in E}m_i<\infty$. Then
\begin{equation*}
\lim_{t\rightarrow\infty}\frac{R_t}{t} = u,\;\;\;\textrm{a.s.}
\end{equation*}
For any $\eta\geq 0$, let $U^\eta = \{U^\eta_t:t\geq 0\}$ be the process defined as
\begin{equation*}
U^\eta_t = \eta^{-1/2}(R_{\eta t}-u\eta t).
\end{equation*}
If $v_i<\infty$ for any $i\in E$, then $U^\eta\Rightarrow\sigma W$ on $D[0,\infty)$ as $\eta\rightarrow\infty$, where $D[0,\infty)$ is the Skorokhod space and $W$ is a standard Brownian motion.
\end{lemma}

The following lemma, which gives the large deviations for Markov renewal processes, is due to Mariani and Zambotti \cite{mariani2012large}.
\begin{lemma}\label{LDP}
For any $t\geq 0$, let $Q_t\in C(E\times E,[0,\infty))$ be the empirical flow defined as
\begin{equation}\label{flow}
Q_t(i,j) = \frac{1}{t}\sum_{n=1}^{N_t}1_{\{\xi_{n-1}=i,\xi_n=j\}}.
\end{equation}
Then the law of $Q_t$ satisfies a large deviation principle with rate $t$ and good rate function $I: C(E\times E,[0,\infty))\rightarrow[0,\infty]$. Moreover, the rate function $I$ is convex.
\end{lemma}

The explicit expression of the rate function $I$ is rather complicated. Readers who are interested in this expression may refer to Equations (5)-(7) in \cite{mariani2012large}.

\subsection{Strong law of large numbers for sample circulations}
Let $X$ be a diffusion process solving the stochastic differential equation \eqref{SDE}, where $b:\mathbb{R}\rightarrow\mathbb{R}$ and $\sigma:\mathbb{R}\rightarrow(0,\infty)$ are continuous functions with period 1. For further references, we introduce some notations. For any $t\geq 0$, let $N_t = \inf\{n\geq 0: T_{n+1}>t\}$ be the number of cycles formed by $X$ up to time $t$ and let
\begin{equation*}
N^+_t = \sum_{n=1}^{N_t}1_{\{X_{T_n}-X_{T_{n-1}}=1\}},\;\;\;
N^-_t = \sum_{n=1}^{N_t}1_{\{X_{T_n}-X_{T_{n-1}}=-1\}}
\end{equation*}
be the numbers of the forward and backward cycles formed by $X$ up to time $t$, respectively.

\begin{definition}\label{empirical}
The sample circulations $J^+_t$ and $J^-_t$ along the forward and backward cycles up to time $t$ are defined as
\begin{equation}\label{limits}
J^+_t = \frac{1}{t}N^+_t,\;\;\;J^-_t = \frac{1}{t}N^-_t,
\end{equation}
respectively. The sample net circulation $J_t$ of $X$ up to time $t$ is defined as $J_t = J^+_t-J^-_t$.
\end{definition}

The following theorem gives the strong law of large numbers for the sample circulations. Recall that if we regard $X$ as a diffusion process on $S^1$, then $X$ is always ergodic with respect to its unique invariant distribution \cite{guang1982invariant}.
\begin{theorem}
Let $\rho$ be the invariant distribution of $X$. Then for any initial distribution of $X$,
\begin{equation}\label{circulations}
\lim_{t\rightarrow\infty}J^+_t = J^+,\;\;\;
\lim_{t\rightarrow\infty}J^-_t = J^-,\;\;\;\textrm{a.s.},
\end{equation}
where $J^+$ and $J^-$ are two positive constants satisfying
\begin{equation*}
\frac{J^+}{J^-} = \exp\left(2\int_0^1\frac{b(x)}{\sigma^2(x)}dx\right),\;\;\;J^+-J^- = \int_{S^1}b(x)\rho(dx),
\end{equation*}
with $b$ viewed as a function on $S^1$ in the second equation.
\end{theorem}

\begin{proof}
By Theorem \ref{renewal}, the cycle dynamics of $X$ is a renewal random walk and thus $N_t$ is a renewal process. By the strong law of large numbers and elementary renewal theorem, we obtain that
\begin{equation}\label{elementary}
J^+ = \lim_{t\rightarrow\infty}\frac{1}{t}N^+_t = \lim_{t\rightarrow\infty}\frac{N_t}{t}\frac{1}{N_t}\sum_{n=1}^{N_t}1_{\{X_{T_n}-X_{T_{n-1}}=1\}} = \frac{P(X_T-X_0=1)}{ET},\;\;\;\textrm{a.s.}
\end{equation}
By Theorem \ref{general}, we have
\begin{equation*}
\frac{J^+}{J^-} = \frac{P(X_T-X_0=1)}{P(X_T-X_0=-1)} = \frac{P(T^+<T^-)}{P(T^-<T^+)} = \exp\left(2\int_0^1\frac{b(x)}{\sigma^2(x)}dx\right).
\end{equation*}
Since $X$ is the solution to the stochastic differential equation \eqref{SDE}, we have
\begin{equation}\label{inter1}
\frac{X_t}{t} = \frac{X_0}{t} + \frac{1}{t}\int_0^tb(X_s)ds + \frac{1}{t}\int_0^t\sigma(X_s)dW_s.
\end{equation}
It follows from Birkhoff's ergodic theorem that
\begin{equation}\label{Birkhoff1}
\lim_{t\rightarrow\infty}\frac{1}{t}\int_0^tb(X_s)ds = \int_{S^1}b(x)\rho(dx),\;\;\;\textrm{a.s.}
\end{equation}
Let $M_t = \int_0^t\sigma(X_s)dW_s$. Since $M$ is a continuous martingale, there exists a Brownian motion $B$ such that $M_t = B_{[M,M]_t}$. By Birkhoff's ergodic theorem again, we obtain that
\begin{equation}\label{Birkhoff2}
\lim_{t\rightarrow\infty}\frac{[M,M]_t}{t} = \lim_{t\rightarrow\infty}\frac{1}{t}\int_0^t\sigma^2(X_s)ds = \int_{S^1}\sigma^2(x)\rho(dx) > 0,\;\;\;\textrm{a.s.}
\end{equation}
By Khinchin's law of the iterated logarithm, it is easy to check that
\begin{equation*}
\lim_{t\rightarrow\infty}\frac{B_t}{t} = 0,\;\;\;\textrm{a.s.}
\end{equation*}
This fact, together with \eqref{Birkhoff1} and \eqref{Birkhoff2}, shows that
\begin{equation}\label{inter2}
\lim_{t\rightarrow\infty}\frac{1}{t}\int_0^t\sigma(X_s)dW_s = \lim_{t\rightarrow\infty}\frac{B_{[M,M]_t}}{t} = \lim_{t\rightarrow\infty}\frac{B_{[M,M]_t}}{[M,M]_t}\cdot\frac{[M,M]_t}{t} = 0,\;\;\;\textrm{a.s.}
\end{equation}
In addition, it is easy to see that $N^+_t-N^-_t = X_{T_{N_t}}-X_0$ and $|X_{T_{N_t}}-X_t|\leq 1$. This implies that
\begin{equation}
\left|J_t-\frac{X_t}{t}\right|\leq\frac{1}{t}.
\end{equation}
Combining \eqref{inter1}, \eqref{Birkhoff1}, and \eqref{inter2}, we finally obtain that
\begin{equation}
J = \lim_{t\rightarrow\infty}J_t = \lim_{t\rightarrow\infty}\frac{X_t}{t} = \int_{S^1}b(x)\rho(dx),\;\;\;\textrm{a.s.}.
\end{equation}
This completes the proof of this theorem.
\end{proof}

\begin{definition}
The limits $J^+$ and $J^-$ in \eqref{circulations} are called the circulations of $X$ along the forward and backward cycles, respectively. The net circulation $J$ of $X$ is defined as $J = J^+-J^-$.
\end{definition}

Intuitively, $J^+$ and $J^-$ represent the numbers of the forward and backward cycles formed by $X$ per unit time, respectively, and $J$ represents the net number of cycles formed by $X$ per unit time. Due to this reason, the net circulation $J$ is also called the rotation number \cite{guang1982invariant}.

The following definition originates from nonequilibrium statistical physics.
\begin{definition}\label{affinity}
The affinity $\gamma$ of $X$ is defined as
\begin{equation*}
\gamma = \log\frac{J^+}{J^-} = 2\int_0^1\frac{b(y)}{\sigma^2(y)}dy.
\end{equation*}
\end{definition}

\subsection{Functional central limit theorem for sample circulations}
\begin{lemma}\label{moments}
For any $p>0$,
\begin{equation*}
ET^p<\infty.
\end{equation*}
\end{lemma}

\begin{proof}
Without loss of generality, we assume that $X$ starts from 0. It is easy to see that the scale function $s$ of $X$ is an strictly increasing $C^2$ function with $s(0) = 0$. By It\^{o}'s formula, it is easy to check that $M = s(X)$ is a diffusion process solving the stochastic differential equation
\begin{equation*}
dM_t = f(M_t)dW_t,\;\;\;Y_0 = 0.
\end{equation*}
where $f = (s'\sigma)\circ s^{-1}$. Let $\sigma_z$ be the hitting time of $z$ by $s(X)$. Then $T = \sigma_{s(-1)}\wedge\sigma_{s(1)}$. Since the distribution of $T$ only depends on the value of $f$ on $[s(-1),s(1)]$, we can assume that $f\geq a$ for some constant $a>0$. Since $M$ is a continuous local martingale, there exists a Brownian motion $B$ such that $M_t = B_{[M,M]_t}$, where
\begin{equation*}
[M,M]_t = \int_0^tf^2(M_s)ds \geq a^2t.
\end{equation*}
Let $\tau_z$ be the hitting time of $z$ by $B$. It is easy to check that $T\leq\tau_{s(-1)}\wedge\tau_{s(1)}/a^2$. This shows that
\begin{equation*}
ET^p \leq \frac{1}{a^{2p}}E[\tau_{s(-1)}\wedge\tau_{s(1)}]^p<\infty.
\end{equation*}
This completes the proof of this lemma.
\end{proof}

The functional central limit theorem for the sample circulations is stated in the next theorem.
\begin{theorem}\label{FCLTcirculation}
For any $\eta\geq 0$, let $J^{\eta,+} = \{J^{\eta,+}_t:t\geq 0\}$, $J^{\eta,-} = \{J^{\eta,+}_t:t\geq 0\}$, $J^\eta = \{J^\eta_t:t\geq 0\}$ be three processes defined as
\begin{equation*}
J^{\eta,+}_t = \eta^{-1/2}(N^+_{\eta t}-J^+\eta t),\;\;\;
J^{\eta,-}_t = \eta^{-1/2}(N^-_{\eta t}-J^-\eta t),\;\;\;
J^\eta_t = J^{\eta,+}_t-J^{\eta,-}_t,
\end{equation*}
respectively. Then $(J^{\eta,+},J^{\eta,-},J^\eta)'\Rightarrow\Sigma W$ on $D([0,\infty),\mathbb{R}^3)$ as $\eta\rightarrow\infty$, where $D([0,\infty),\mathbb{R}^3)$ is the Skorokhod space, $W$ is a standard three-dimensional Brownian motion, and
\begin{equation*}
\Sigma\Sigma' = m^{-1}
\begin{pmatrix}
\rho_++\alpha\rho_+^2 & 2\alpha\rho_-\rho_+ & 2\rho_++2\alpha\rho_+(\rho_+-\rho_-) \\
2\alpha\rho_-\rho_+ & \rho_-+\alpha\rho_-^2 & -2\rho_-+2\alpha\rho_-(\rho_+-\rho_-) \\
2\rho_++2\alpha\rho_+(\rho_+-\rho_-) & -2\rho_-+2\alpha\rho_-(\rho_+-\rho_-) & 1+\alpha(\rho_+-\rho_-)^2
\end{pmatrix}
\end{equation*}
with $m = ET$, $\alpha = ET^2/(ET)^2-2$, $\rho_+ = P(T^+<T^-)$, and $\rho_- = P(T^-<T^+)$.
\end{theorem}

\begin{proof}
Let $K = (J^{\eta,+},J^{\eta,-},J^\eta)'$. By \cite[Chapter 29, Theorem 29.16]{davidson1994stochastic}, the desired result holds if and only if $\lambda'K\Rightarrow\lambda'\Sigma W$ for any row vector $\lambda = (a,b,c)\in\mathbb{R}^3$. Note that
\begin{equation*}
\lambda'K_t = aJ^{\eta,+}_t+bJ^{\eta,-}_t+cJ^\eta_t = \eta^{-1/2}(R_{\eta t}-u\eta t),
\end{equation*}
where $u = aJ^++bJ^-+cJ$ and
\begin{equation*}
R_t = aN^+_t+bN^-_t+c(N^+_t-N^-_t) = \sum_{n=1}^{N_t}(a+c)1_{\{\xi_n=1\}}+(b-c)1_{\{\xi_n=-1\}}
\end{equation*}
is a reward process with $r_{-1,-1} = r_{1,-1} = b-c$ and $r_{-1,1} = r_{1,1} = a+c$. By Lemma \ref{moments}, it is easy to see that $m_{-1} = m_1 = ET<\infty$ and $v_{-1} = v_1 = ET^2<\infty$. By Lemma \ref{FCLT}, it is straightforward to check that $\lambda'K\Rightarrow\lambda'\Sigma W$.
\end{proof}

\subsection{Large deviations for sample circulations}
The large deviation principle for the sample circulations is stated in the following theorem. For the precise definition of the large deviation principle for a family of probability measures, please refer to \cite[Page 3]{varadhan1984large}.
\begin{theorem}\label{LDPcirculation}
The law of $(J^+_t,J^-_t,J_t)$ satisfies a large deviation principle with rate $t$ and good rate function $I_0:\mathbb{R}^3\rightarrow[0,\infty]$. Moreover, $I_0$ is convex and $I_0(J^+,J^-,J) = 0$.
\end{theorem}

\begin{proof}
By Theorem \ref{renewal}, $L$ is a renewal random walk and thus can be viewed as a Markov renewal process. Let $\{Q_t:t\geq 0\}$ be the empirical flow of $L$ defined in \eqref{flow}. Then
\begin{equation*}
J^+_t = \frac{1}{t}\sum_{n=1}^{N_t}1_{\{\xi_n=1\}} = \sum_{i\in E}Q_t(i,1),\;\;\;
J^-_t = \frac{1}{t}\sum_{n=1}^{N_t}1_{\{\xi_n=-1\}} = \sum_{i\in E}Q_t(i,-1).
\end{equation*}
We define a continuous map $F:C(E\times E,[0,\infty))\rightarrow\mathbb{R}^3$ as
\begin{equation}\label{contraction}
F(Q) = \left(\sum_{i\in E}Q(i,1),\sum_{i\in E}Q(i,-1),\sum_{i\in E}(Q(i,1)-Q(i,-1))\right).
\end{equation}
Then we have $(J^+_t,J^-_t,J_t) = F(Q_t)$. By Lemma \ref{LDP}, the law of $Q_t$ satisfies a large deviation principle with rate $t$ and good rate function $I$. By the contraction principle, the law of $(J^+_t,J^-_t,J_t)$ satisfies a large deviation principle with rate $t$ and good rate function $I_0:\mathbb{R}^3\rightarrow[0,\infty]$ defined as
\begin{equation}\label{ratefunction}
I_0(x) = \inf_{Q\in F^{-1}(x)}I(Q).
\end{equation}

We next prove that $I_0$ is convex. We make a crucial observation that the map $F$ is linear. Moreover, it follows from Lemma \ref{LDP} that $I$ is convex. These two facts suggest that for any $\lambda_1,\lambda_2,\lambda_3\geq 0$ satisfying $\lambda_1+\lambda_2+\lambda_3 = 1$ and $x_1,x_2,x_3\in\mathbb{R}^3$,
\begin{equation*}
\begin{split}
& I_0(\lambda_1x_1+\lambda_2x_2+\lambda_3x_3) \leq \inf_{Q\in\lambda_1F^{-1}(x_1)+\lambda_2F^{-1}(x_2)+\lambda_3F^{-1}(x_3)}I(Q) \\
&= \inf_{Q_1\in F^{-1}(x_1),Q_2\in F^{-1}(x_2),Q_3\in F^{-1}(x_3)}I(\lambda_1Q_1+\lambda_2Q_2+\lambda_3Q_3) \\
&= \lambda_1\inf_{Q\in F^{-1}(x_1)}I(Q)+\lambda_2\inf_{Q\in F^{-1}(x_2)}I(Q)+\lambda_3\inf_{Q\in F^{-1}(x_3)}I(Q) \\
&= \lambda_1I_0(x_1)+\lambda_2I_0(x_2)+\lambda_3I_0(x_3),
\end{split}
\end{equation*}
which shows that $I_0$ is also convex.

We finally prove that $I_0(J^+,J^-,J) = 0$. Let $B_\epsilon = \{y\in\mathbb{R}^3: |y-(J^+,J^-,J)|\leq\epsilon\}$. Since $I_0$ is lower semi-continuous and $(J^+_t,J^-_t,J_t)\rightarrow(J^+,J^-,J)$, a.s., we obtain that
\begin{equation*}
\begin{split}
I_0(J^+,J^-,J) \leq \lim_{\epsilon\rightarrow 0}\inf_{y\in B_\epsilon}I_0(y)
\leq -\limsup_{\epsilon\rightarrow 0}\limsup_{t\rightarrow\infty}\frac{1}{t}\log P((J^+_t,J^-_t,J_t)\in B_\epsilon) = 0.
\end{split}
\end{equation*}
This clearly shows that $I_0(J^+,J^-,J)=0$.
\end{proof}

Since the cycle dynamics of $X$ is simply a renewal random walk, where the interarrival times and jump probability are both independent of the initial distribution of $X$. This suggests that the rate function $I_0$ is also independent of the initial distribution of $X$.

\section{Fluctuation theorems for diffusion processes on the circle}
In this section, we shall use the cycle symmetry to prove various types of FTs for diffusion processes on the circle.

\subsection{Fluctuation theorems for sample circulations}
The next proposition characterizes the symmetry of the distribution of the sample circulations. Results of the following type are called transient FTs in nonequilibrium statistical physics.
\begin{proposition}\label{distribution}
Let $\mathbb{N} = \{0,1,2,\cdots\}$. For each $t\geq 0$ and any $m,n\in\mathbb{N}$,
\begin{equation*}
\frac{P(N^+_t=n,N^-_t=m)}{P(N^+_t=m,N^-_t=n)} = e^{\gamma(n-m)},
\end{equation*}
where $\gamma$ is the affinity of $X$.
\end{proposition}

\begin{proof}
With the notations in Theorem \ref{renewal}, we have
\begin{equation*}
P(N^+_t=n,N^-_t=m) = \sum_{s_1,\cdots,s_{n+m}}P(T_{n+m}\leq t<T_{n+m+1},\xi_1=s_1,\cdots,\xi_{n+m}=s_{n+m}),
\end{equation*}
where $s_1,\cdots,s_{n+m}$ ranges over all choices such that $n$ of them equal $1$ and $m$ of them equal $-1$. By Theorem \ref{renewal}, it is easy to see that
\begin{equation*}
P(T_{n+m}\leq t<T_{n+m+1},\xi_1=s_1,\cdots,\xi_{n+m}=s_{n+m}) = P(T_{n+m}\leq t<T_{n+m+1})\rho_+^n\rho_-^m,
\end{equation*}
where $\rho_+ = P(T^+<T^-)$ and $\rho_- = P(T^-<T^+)$. By Theorem \ref{general}, we obtain that
\begin{equation*}
\begin{split}
& P(N^+_t=n,N^-_t=m) = \sum_{s_1,\cdots,s_{n+m}}P(T_{n+m}\leq t<T_{n+m+1})\rho_+^n\rho_-^m \\
&= P(T_{n+m}\leq t<T_{n+m+1})C_{n+m}^n\rho_+^m\rho_-^ne^{\gamma(n-m)}
= P(N^+_t=m,N^-_t=n)e^{\gamma(n-m)},
\end{split}
\end{equation*}
which gives the desired result.
\end{proof}

The following lemma gives a lower bound for the survival function of the cycle forming time $T$. In order not to interrupt things, we defer the proof of this lemma to the final section of this paper.
\begin{lemma}\label{estimation}
Let $T$ be the cycle forming time of $X$. Then there exists $\rho>0$, such that for any $t\geq 0$,
\begin{equation*}
P(T>t) \geq e^{-\rho t}.
\end{equation*}
\end{lemma}

The next result follows directly from the above lemma.
\begin{lemma}\label{finite}
There exists $\rho>0$, such that for any $\beta\geq0$ and $t\geq 0$,
\begin{equation*}
Ee^{\beta N_t} \leq e^{(e^\beta-1)\rho t}.
\end{equation*}
\end{lemma}

\begin{proof}
By Lemma \ref{estimation}, there exists $\rho>0$ such that $P(T>t) \geq e^{-\rho t}$. This shows that the cycle forming time $T$ is stochastically larger than an exponential random variable with rate $\rho$. Thus it is easy to see that the $n$th cycle forming time $T_n$ is stochastically larger than the sum of $n$ independent exponential random variables with rate $\rho$. This further suggests that $N_t$ is stochastically dominated by a Poisson random variable $R_t$ with parameter $\rho t$. Thus we obtain that
\begin{equation*}
\begin{split}
& Ee^{\beta N_t}
= \int_{-\infty}^\infty\beta e^{\beta x}P(N_t\geq x)dx
\leq \int_{-\infty}^\infty\beta e^{\beta x}P(R_t\geq x)dx \\
&= Ee^{\beta R_t} = \sum_{n=0}^\infty e^{\beta n}\frac{(\rho t)^n}{n!}e^{-\rho t}
= \exp\left((e^{\beta}-1)\rho t\right),
\end{split}
\end{equation*}
which gives the desired result.
\end{proof}

The next proposition characterizes the symmetry of the moment generating function of the sample circulations. Results of the following type are called Kurchan-Lebowitz-Spohn-type FTs in nonequilibrium statistical physics.
\begin{proposition}\label{generating}
Let
\begin{equation*}
g_t(\lambda_1,\lambda_2) = Ee^{\lambda_1N^+_t+\lambda_2N^-_t} = Ee^{t(\lambda_1J^+_t+\lambda_2J^-_t)}.
\end{equation*}
Then for each $t\geq 0$ and any $\lambda_1,\lambda_2\in\mathbb{R}$, we have $g_t(\lambda_1,\lambda_2)<\infty$ and
\begin{equation*}
g_t(\lambda_1,\lambda_2) = g_t(\lambda_2-\gamma,\lambda_1+\gamma).
\end{equation*}
\end{proposition}

\begin{proof}
By Lemma \ref{finite}, we obtain that
\begin{equation}\label{estimate}
g_t(\lambda_1,\lambda_2)
\leq Ee^{|\lambda_1|N^+_t+|\lambda_2|N^-_t}
\leq Ee^{\alpha(N^+_t+N^-_t)}
= Ee^{\alpha N_t}<\infty,
\end{equation}
where $\alpha=\max\{|\lambda_1|,|\lambda_2|\}$. Moreover, it follows from Proposition \ref{distribution} that
\begin{equation*}
\begin{split}
& g_t(\lambda_1,\lambda_2) = \sum_{n,m\in\mathbb{N}}e^{\lambda_1n+\lambda_2m}P(N^+_t=n,N^-_t=m) \\
&= \sum_{n,m\in\mathbb{N}}e^{(\lambda_1+\gamma)n+(\lambda_2-\gamma)m}P(N^+_t=m,N^-_t=n)
= g_t(\lambda_2-\gamma,\lambda_1+\gamma),
\end{split}
\end{equation*}
which gives the desired result.
\end{proof}

The next theorem characterizes the symmetry of the rate function of the sample circulations. Theorems of the following type are called Gallavotti-Cohen-type FTs in nonequilibrium statistical physics. Recall that the Legendre-Fenchel transform of a function $f:\mathbb{R}^n\rightarrow[-\infty,\infty]$ is defined as
\begin{equation*}
f^*(\lambda) = \sup_{x\in\mathbb{R}^n}\{\lambda\cdot x-f(x)\}.
\end{equation*}
and a function $f:\mathbb{R}^n\rightarrow(-\infty,\infty]$ is called proper if $f(x)<\infty$ for at least one $x$.
\begin{theorem}\label{symmetry}
The law of $(J^+_t,J^-_t)$ satisfies a large deviation principle with rate $t$ and good rate function $I_1:\mathbb{R}^2\rightarrow[0,\infty]$, which is convex and satisfies $I_1(J^+,J^-) = 0$. Moreover, the rate function $I_1$ has the following symmetry: for any $x_1,x_2\in\mathbb{R}$,
\begin{equation*}
I_1(x_1,x_2) = I_1(x_2,x_1)-\gamma(x_1-x_2).
\end{equation*}
\end{theorem}

\begin{proof}
By Theorem \ref{LDPcirculation} and the contraction principle, the law of $(J^+_t,J^-_t)$ satisfies a large deviation principle with rate $t$ and good rate function $I_1:\mathbb{R}^2\rightarrow[0,\infty]$. The proof of the facts that $I_1$ is convex and $I_1(J^+,J^-)=0$ follows the same line as that of Theorem \ref{LDPcirculation}.

By a strengthened version of Varadhan's lemma \cite[Theorem 4.3.1]{dembo2010large}, if there exists $\beta>1$ such that the following moment condition is satisfied:
\begin{equation*}
\limsup_{t\rightarrow\infty}\frac{1}{t}\log Ee^{\beta t(\lambda_1J^+_t+\lambda_2J^-_t)} < \infty,
\end{equation*}
then for any $\lambda_1,\lambda_2\in\mathbb{R}$,
\begin{equation}\label{varadhan}
\lim_{t\rightarrow\infty}\frac{1}{t}\log g_t(\lambda_1,\lambda_2) = I_1^*(x_1,x_2).
\end{equation}
By Lemma \ref{finite}, we obtain that
\begin{equation*}
Ee^{\beta t(\lambda_1J^+_t+\lambda_2J^-_t)} \leq Ee^{\beta|\lambda_1|N^+_t+\beta|\lambda_2|N^-_t} \leq Ee^{\beta\alpha N_t} \leq \exp\left((e^{\beta\alpha}-1)\rho t\right),
\end{equation*}
where $\alpha=\max\{|\lambda_1|,|\lambda_2|\}$. This shows that for any $\beta\geq 0$,
\begin{equation*}
\limsup_{t\rightarrow\infty}\frac{1}{t}\log Ee^{\beta t(\lambda_1J^+_t+\lambda_2J^-_t)}
\leq\limsup_{t\rightarrow\infty}\frac{1}{t}\log Ee^{\beta\alpha N_t} \leq (e^{\beta\alpha}-1)\rho <\infty.
\end{equation*}
Thus we have proved \eqref{varadhan}. It thus follows from \eqref{varadhan} and Proposition \ref{generating} that $I_1^*(\lambda_1,\lambda_2) = I_1^*(\lambda_2-\gamma,\lambda_1+\gamma)$. This implies that
\begin{equation}\label{involution}
\begin{split}
& I_1^{**}(x_1,x_2)
= \sup_{\lambda_1,\lambda_2\in\mathbb{R}}\{\lambda_1x_1+\lambda_2x_2- I_1^*(\lambda_2-\gamma,\lambda_1+\gamma)\} \\
&= \sup_{\lambda_1,\lambda_2\in\mathbb{R}}\{(\lambda_1-\gamma)x_1+(\lambda_2+\gamma)x_2- I_1^*(\lambda_2,\lambda_1)\}
= I_1^{**}(x_2,x_1)-\gamma(x_1-x_2).
\end{split}
\end{equation}
On the other hand, the Fenchel-Moreau theorem \cite[Theorem 4.2.1]{borwein2010convex} shows that if a function $f$ is proper, then $f^{**}=f$ if and only if $f$ is convex and lower semi-continuous. By Theorem \ref{LDPcirculation}, $I_1$ is a good rate function which is also convex. This shows that $I_1$ is proper, convex, and lower semi-continuous. It thus follows from the Fenchel-Moreau theorem that $I_1=I_1^{**}$. This fact, together with \eqref{involution}, gives the desired result.
\end{proof}

\subsection{Fluctuation theorems for sample net circulation}
The following transient FT characterizes the symmetry of the distribution of the sample net circulation.
\begin{proposition}
For each integer $k$, we have
\begin{equation}
\frac{P(N_t=k)}{P(N_t=-k)} = e^{\gamma k}.
\end{equation}
\end{proposition}

\begin{proof}
By Proposition \ref{distribution}, we have
\begin{equation*}
\begin{split}
& P(N_t=k) = \sum_{n-m=k}P(N^+_t=n,N^-_t=m) \\
&= \sum_{n-m=k}P(N^+_t=m,N^-_t=n)e^{\gamma(n-m)} = P(N_t=-k)e^{\gamma k},
\end{split}
\end{equation*}
which gives the desired result.
\end{proof}

The following Kurchan-Lebowitz-Spohn-type FT characterizes the symmetry of the moment generating function of the sample net circulation.
\begin{proposition}\label{KLS}
For each $t\geq 0$ and $\lambda\in\mathbb{R}$, we have $Ee^{\lambda N_t}<\infty$ and
\begin{equation*}
Ee^{\lambda N_t} = Ee^{-(\lambda+\gamma)N_t}.
\end{equation*}
\end{proposition}

\begin{proof}
If we take $\lambda_1=\lambda$ and $\lambda_2=-\lambda$ in Proposition \ref{generating}, then we have $Ee^{\lambda N_t}<\infty$ and
\begin{equation*}
Ee^{\lambda N_t} = g_t(\lambda,-\lambda)
= g_t(-(\lambda+\gamma),\lambda+\gamma)
= Ee^{-(\lambda+\gamma)N_t},
\end{equation*}
which is just the desired result.
\end{proof}

Results of the following type are called integral FTs in nonequilibrium statistical physics.
\begin{corollary}\label{integral}
For each $t\geq 0$,
\begin{equation*}
Ee^{-\gamma N_t} = 1.
\end{equation*}
\end{corollary}

\begin{proof}
If we take $\lambda=-\gamma$ in Proposition \ref{KLS}, then we obtain the desired result.
\end{proof}

The following Gallavotti-Cohen-type FT characterizes the symmetry of the rate function of the sample net circulation.
\begin{theorem}\label{LDPnetcirculation}
The law of $J_t$ satisfies a large deviation principle with rate $t$ and good rate function $I_2:\mathbb{R}\rightarrow[0,\infty]$, which is convex and satisfies $I_2(J) = 0$. Moreover, the rate function $I_2$ has the following symmetry: for each $x\in\mathbb{R}$,
\begin{equation*}
I_2(x) = I_2(-x)-\gamma x.
\end{equation*}
\end{theorem}

\begin{proof}
By Theorem \ref{symmetry} and the contraction principle, the law of $J_t$ satisfies a large deviation principle with rate $t$ and good rate function
\begin{equation*}
I_2(x) = \inf_{x_1-x_2=x}I_1(x_1,x_2).
\end{equation*}
The proof of the facts that $I_2$ is convex and $I_2(J)=0$ follows the same line as that of Theorem \ref{LDPcirculation}. By Theorem \ref{symmetry}, we have
\begin{equation*}
I_2(x) = \inf_{x_1-x_2=x}[I_1(x_2,x_1)-\gamma(x_1-x_2)] = \inf_{x_2-x_1=-x}I_1(x_2,x_1)-\gamma x = I_2(-x)-\gamma x.
\end{equation*}
This completes the proof of this theorem.
\end{proof}

\begin{remark}
It is easy to see that the FTs for the sample circulations and net circulation can be easily extended to general renewal random walks whenever the distribution of the interarrival times satisfies the inequality in Lemma \ref{estimation}.
\end{remark}

\subsection{Relationship with reversibility}
The FTs discussed above are closely related to the reversibility for diffusion processes on the circle. The following theorem gives several equivalent conditions for $X$ to be a symmetric Markov process, where $X$ should be viewed as a diffusion process on $S^1$ rather than on $\mathbb{R}$.
\begin{theorem}\label{reversibility}
Let $\rho$ be the invariant distribution of $X$. Then the following statements are equivalent: \\
(i) $X$ is a symmetric Markov process with respect to $\rho$; \\
(ii) $\gamma = 0$; \\
(iii) $J = 0$; \\
(iv) $I_1(x_1,x_2) = I_1(x_2,x_1)$ for any $x_1,x_2\in\mathbb{R}$; \\
(v) $I_2(x) = I_2(-x)$ for any $x\in\mathbb{R}$.
\end{theorem}

\begin{proof}
It is a classical result that (i) and (ii) are equivalent \cite{guang1982invariant, jiang2004mathematical}. By the definition of the affinity $\gamma$, it is easy to see that $\gamma = 0$ if and only if $J^+=J^-$. This shows that (ii) and (iii) are equivalent. By Theorem \ref{symmetry} and Theorem \ref{LDPnetcirculation}, it is easy to see that (ii),(iv), and (v) are equivalent.
\end{proof}

\subsection{Fluctuation theorems for sample entropy production rate}
Entropy production rate is a central concept in nonequilibrium statistical physics. In the previous work, the transient and integral FTs for the sample entropy production rate have been established for pretty general stochastic processes \cite{ge2007transient}. However, it turns out to be rather difficult to establish the large deviation principle and Gallavotti-Cohen-type FT for the sample entropy production rate. Here, we shall give a simple proof of the Gallavotti-Cohen-type FT for the sample entropy production rate of diffusion processes on the circle.

Let $X = \{X_t: t\geq 0\}$ be a diffusion process on $S^1$ with smooth diffusion coefficient $a:S^1\rightarrow(0,\infty)$ and smooth drift $b:S^1\rightarrow\mathbb{R}$. In fact, $X$ is a Brownian motion with drift on the Riemannian manifold $S^1$ with some changed Riemannian metric \cite[Chapter 5, Page 121]{jiang2004mathematical}. Thus $X$ has a unique invariant distribution $\rho$, which has a strictly positive smooth density $\eta(x)$ with respect to the volume element of $S^1$ \cite[Chapter V, Proposition 4.5]{ikeda1989stochastic}.

\begin{definition}
Let $Y$ be the lifted process of $X$ on $\mathbb{R}$ with diffusion coefficient $a:\mathbb{R}\rightarrow(0,\infty)$ and drift $b:\mathbb{R}\rightarrow\mathbb{R}$. The sample entropy production rate $E_t$ of $X$ up to time $t$ is defined as
\begin{equation*}
E_t = \frac{1}{t}\int_0^t\left[\frac{2b}{a}-(\log(a\eta))'\right](Y_s)\circ dY_s,
\end{equation*}
where $\eta$ is viewed as a periodic function on $\mathbb{R}$ and the symbol $\circ$ means that the stochastic integral is taken in the Stratonovich sense.
\end{definition}

\begin{remark}
If $X$ is a stationary diffusion process on $S^1$, there is still another way to defined the sample entropy production rate $E_t$. Let $\mu^+_t$ denote the distribution of the process $\{X_s:0\leq s\leq t\}$ and let $\mu^-_t$ denote the distribution of its timed-reversal $\{X_{t-s}:0\leq s\leq t\}$. Let $C([0,t],S^1)$ denote the space of all continuous functions on $[0,t]$ with values in $S^1$. Then $\mu^+_t$ and $\mu^-_t$ are both probability measures on $C([0,t],S^1)$. With these notations, the sample entropy production rate $E_t$ can be represented as the logarithm of the Radon-Nikodym derivative between the original process and its time-reversal \cite[Proposition 5.3.5]{jiang2004mathematical}:
\begin{equation*}
E_t(\omega) = \frac{1}{t}\log\frac{d\mu^+_t}{d\mu^-_t}(X_\cdot(\omega)).
\end{equation*}
\end{remark}

The sample entropy production rate and net circulation are related as follows.
\begin{lemma}\label{close}
\begin{equation*}
|E_t-J_t\gamma| = O\left(\frac{1}{t}\right).
\end{equation*}
\end{lemma}

\begin{proof}
Let
\begin{equation*}
F(x) = \int_0^x\frac{2b(y)}{a(y)}dy-\log(a(x)\eta(x)).
\end{equation*}
By It\^{o}'s formula, we obtain that
\begin{equation*}
E_t = \frac{1}{t}\int_0^tF'(Y_s)\circ dY_s = \frac{1}{t}[F(Y_{T_{N_t}})-F(Y_0)] + \frac{1}{t}[F(Y_t)-F(Y_{T_{N_t}})].
\end{equation*}
Note that $Y_{T_{N_t}} = Y_0+N^+_t-N^-_t$. By the periodicity of $a$, $b$, and $\eta$, we have
\begin{equation*}
\frac{1}{t}[F(Y_{T_{N_t}})-F(Y_0)] = \frac{1}{t}(N^+_t-N^-_t)\gamma = J_t\gamma.
\end{equation*}
Since $|Y_{T_{N_t}}-Y_t|\leq 1$, we obtain that
\begin{equation*}
\begin{split}
& \frac{1}{t}|F(Y_t)-F(Y_{T_{N_t}})|
= \left|\frac{1}{t}\int_{Y_{T_{N_t}}}^{Y_t}\frac{2b(y)}{a(y)}dy - \frac{1}{t}\log(a(Y_t)\eta(Y_t)) + \frac{1}{t}\log(a(Y_{T_{N_t}})\eta(Y_{T_{N_t}}))\right| \\
&\leq \left(\left\|\frac{2b}{a}\right\|_\infty+2\|\log(a\eta)\|_\infty\right)\frac{1}{t}.
\end{split}
\end{equation*}
This completes the proof of this lemma.
\end{proof}

\begin{remark}
The entropy production rate of diffusion processes on the circle has been extensively studied in the previous work \cite{guang1982invariant, jiang2004mathematical}. In fact, the entropy production rate $e$ of $X$ has the form of
\begin{equation*}
e = (J^+-J^-)\log\frac{J^+}{J^-} = J\gamma.
\end{equation*}
According to the above lemma, it is easy to see that the entropy production rate $e$ is the almost sure limit of the sample entropy production rate $E_t$:
\begin{equation*}
\lim_{t\rightarrow\infty}E_t = \lim_{t\rightarrow\infty}J_t\gamma = J\gamma = e,\;\;\;\textrm{a.s.}
\end{equation*}
\end{remark}

The following theorem gives the large deviation principle and Gallavotti-Cohen-type FT for the sample entropy production rate.
\begin{theorem}
The law of $E_t$ satisfies a large deviation principle with rate $t$ and good rate function $I_3:\mathbb{R}\rightarrow[0,\infty]$ which has the following form
\begin{equation*}
I_3(x) =
\begin{cases}
0, &\textrm{if}\;\gamma = 0\;\textrm{and}\;x = 0, \\
\infty, &\textrm{if}\;\gamma = 0\;\textrm{and}\;x\neq 0, \\
I_2(x/\gamma), &\textrm{if}\;\gamma\neq 0.
\end{cases}
\end{equation*}
Moreover, the rate function $I_3$ has the following symmetry: for each $x\in\mathbb{R}$,
\begin{equation*}
I_3(x) = I_3(-x)-x.
\end{equation*}
\end{theorem}

\begin{proof}
When $\gamma = 0$, $X$ is reversible and thus $E_t=0$ for any $t\geq 0$. Thus the result holds when $\gamma = 0$. We only need to prove the result when $\gamma\neq 0$. By Theorem \ref{LDPnetcirculation}, the law of $J_t\gamma$ satisfies a large deviation principle with rate $t$ and rate function $I_3$. We shall next prove that the law of $E_t$ also satisfies a large deviation principle with rate $t$ and rate function $I_3$.

We first prove the lower bound. For any $x\in\mathbb{R}$ and $\delta>0$, when $t$ is sufficiently large, we have
\begin{equation}
P(E_t\in B_\delta(x))\geq P(J_t\gamma\in B_{\delta/2}(x))=P(J_t\in B_{\delta/2\gamma}(x/\gamma)),
\end{equation}
where $B_\delta(x) = \{y\in\mathbb{R}:|y-x|<\delta\}$. This shows that
\begin{equation*}
\liminf_{t\rightarrow\infty}\frac{1}{t}\log P(E_t\in B_\delta(x)) \geq \liminf_{t\rightarrow\infty}\frac{1}{t}\log P(J_t\in B_{\delta/2\gamma}(x/\gamma)) \geq -I_2(x/\gamma) = -I_3(x).
\end{equation*}
This gives the lower bound of the large deviation principle.

We next prove the upper bound. For any $\alpha\geq 0$, let $F_3(\alpha) = \{x\in\mathbb{R}:I_3(x)\leq\alpha\}$ be the level set of $I_3$ and let $F_2(\alpha) = \{x\in\mathbb{R}:I_2(x)\leq\alpha\}$ be the level set of $I_2$. It is easy to see that $F_3(\alpha) = \gamma F_2(\alpha)$. Let $d(x,y) = |x-y|$ be the Euclidean metric on $\mathbb{R}$. Then for any $\delta>0$, when $t$ is sufficiently large, we have
\begin{equation}
P(d(E_t,F_3(\alpha))\geq\delta)\leq P(d(J_t\gamma,F_3(\alpha))\geq\delta/2)=P(d(J_t,F_2(\alpha))\geq\delta/2\gamma).
\end{equation}
This shows that
\begin{equation}
\limsup_{t\rightarrow\infty}\frac{1}{t}\log P(d(E_t,F_3(\alpha))\geq\delta)\leq\limsup_{t\rightarrow\infty}\frac{1}{t}\log P(d(J_t,F_2(\alpha))\geq\delta/2\gamma)\leq -\alpha.
\end{equation}
This gives the upper bound of the large deviation principle. Finally, the symmetry of $I_3$ follows from that of $I_2$ given in Theorem \ref{LDPnetcirculation}.
\end{proof}

\section{Appendix}
In this section, we shall give the proof of Lemma \ref{estimation}. To this end, we need the following lemma.
\begin{lemma}\label{prep}
Let $X$ be a diffusion process solving the stochastic differential equation
\begin{equation*}
dX_t = \sigma(X_t)dW_t,\;\;\;X_0 = 0,
\end{equation*}
where $\sigma^2$ is bounded from both below and above. Let $\tau_z$ be the hitting time of $z$ by $X$. Then for any $x<0<y$, there exists $\rho>0$ and $0<\eta\leq 1$, such that for any $t\geq 0$,
\begin{equation*}
P(\tau_x\wedge\tau_y>t) \geq e^{-\rho t}.
\end{equation*}
\end{lemma}

\begin{proof}
Since $\sigma^2$ is bounded from both below and above, $X$ is non-explosive and recurrent \cite[Chapter 5, Proposition 5.22]{karatzas1998brownian}. Thus all the hitting times defined below are finite. Let $S_1 = \inf\{t\geq 1: X_t=0\}$ be the first zero of $X$ after time 1. For each $n\geq 2$, let $S_n = \inf\{t\geq S_{n-1}+1: X_t=0\}$ be the first zero of $X$ after time $S_{n-1}+1$. It is easy to see that $S_n\geq n$ and $X_{S_n} = 0$. By the strong Markov property of $X$, we have
\begin{equation*}
P(\tau_x\wedge\tau_y>n) \geq P(\tau_x\wedge\tau_y>S_n) = P(\tau_x\wedge\tau_y>S_1)^n.
\end{equation*}
By the Markov property of $X$, we have
\begin{equation}\label{temp1}
\begin{split}
&\; P(\tau_x\wedge\tau_y>S_1) = P(\tau_x\wedge\tau_y>1,\;\textrm{$X$ does not hit $x$ and $y$ between 1 and $S_1$}) \\
\geq&\; P(\tau_x\wedge\tau_y>1,\;0<X_1<y/2,\;\textrm{$X$ does not hit $x$ and $y$ between 1 and $S_1$}) \\
&\; + P(\tau_x\wedge\tau_y>1,\;x/2<X_1<0,\;\textrm{$X$ does not hit $x$ and $y$ between 1 and $S_1$}) \\
=&\; E1_{\{\tau_x\wedge\tau_y>1,\;0<X_1<y/2\}}P_{X_1}(\tau_0<\tau_y) + E1_{\{\tau_x\wedge\tau_y>1,\;x/2<X_1<0\}}P_{X_1}(\tau_0<\tau_x).
\end{split}
\end{equation}
It is easy to see that the scale function of $X$ is $s(x) = x$. Thus for each $0<z<y$,
\begin{equation*}
P_z(\tau_0<\tau_y) = \frac{y-z}{y},
\end{equation*}
This implies that
\begin{equation}\label{temp2}
\inf_{0<z<y/2}P_z(\tau_0<\tau_y) = \frac{1}{2},\;\;\;
\inf_{x/2<z<0}P_z(\tau_0<\tau_x) = \frac{1}{2}.
\end{equation}
It thus follows from \eqref{temp1} and \eqref{temp2} that
\begin{equation*}
P(\tau_x\wedge\tau_y>S_1) \geq \frac{1}{2}P(\tau_{x/2}\wedge\tau_{y/2}>1).
\end{equation*}
Since $X$ is a continuous martingale, there exists a Brownian motion $B$, such that $X_t = B_{[X,X]_t}$, where $[X,X]_t = \int_0^t\sigma^2(X_s)ds$. Let $\tilde\tau_z$ be the hitting time of $z$ by $B$. By Girsanov's theorem, it is easy to prove that the support of the Winner measure is the Winner space itself. This suggests that
\begin{equation*}
P(\tau_{x/2}\wedge\tau_{y/2}>1) = P(\tilde\tau_{x/2}\wedge\tilde\tau_{y/2}>[X,X]_1)
\geq P(\tilde\tau_{x/2}\wedge\tilde\tau_{y/2}>\|\sigma\|_\infty^2) > 0.
\end{equation*}
Thus there exists $0<\eta\leq 1$ such that $P(\tau_x\wedge\tau_y>n)\geq\eta^n$. For any $t\geq 0$, choose $n\in\mathbb{N}$, such that $n\leq t<n+1$. Thus we have
\begin{equation}\label{temp}
P(\tau_x\wedge\tau_y>t) \geq P(\tau_x\wedge\tau_y>n+1) \geq \eta^{n+1} \geq \eta^{t+1} = \eta e^{-\rho t},
\end{equation}
where $\rho = -\log\eta>0$.

On the other hand, it is easy to check that
\begin{equation*}
P(\tau_x\wedge\tau_y\leq t) \leq P(\tilde\tau_x\wedge\tilde\tau_y\leq\|\sigma\|_\infty^2t)
\leq P(\tilde\tau_x\leq\|\sigma\|_\infty^2t)+P(\tilde\tau_y\leq\|\sigma\|_\infty^2t).
\end{equation*}
For convenience, let $c = \|\sigma\|_\infty$. Recall that the probability density of $\tilde\tau_x$ is
\begin{equation*}
P(\tilde\tau_x\in du) = \frac{1}{\sqrt{2\pi u^3}}|x|e^{-\frac{x^2}{2u}}du.
\end{equation*}
Thus we have
\begin{equation*}
P(\tilde\tau_x\leq c^2t) = \int_0^{c^2t}\frac{1}{\sqrt{2\pi u^3}}|x|e^{-\frac{x^2}{2u}}du = \sqrt{\frac{2}{\pi}}\int_{\frac{|x|}{c\sqrt{t}}}^\infty e^{-\frac{s^2}{2}}ds
\leq \frac{K}{|x|}\sqrt{t}e^{-\frac{x^2}{2c^2t}},
\end{equation*}
where $K = \sqrt{2/\pi}c$. Thus we obtain that
\begin{equation*}
P(\tau_x\wedge\tau_y>t) \geq 1-\frac{K}{|x|}\sqrt{t}e^{-\frac{x^2}{2c^2t}}-\frac{K}{y}\sqrt{t}e^{-\frac{y^2}{2c^2t}}.
\end{equation*}
For any $\rho_1>0$, let
\begin{equation*}
f(t) = 1-e^{-\rho_1t}-\frac{K}{|x|}\sqrt{t}e^{-\frac{x^2}{2c^2t}}-\frac{K}{y}\sqrt{t}e^{-\frac{y^2}{2c^2t}}.
\end{equation*}
It is easy to see that $f(0)=0$ and $f'(0)>0$. Thus there exists $T>0$, such that $f(t)\geq 0$ for any $0\leq t\leq T$. This suggests that for any $0\leq t\leq T$,
\begin{equation*}
P(\tau_x\wedge\tau_y>t) \geq e^{-\rho_1t}.
\end{equation*}
In view of \eqref{temp}, it is easy to check that there exists $\rho_2>0$ such that for any $t\geq T$,
\begin{equation*}
P(\tau_x\wedge\tau_y>t) \geq e^{-\rho_2t}.
\end{equation*}
The above two equations give the desired result.
\end{proof}

We are now in a position to prove Lemma \ref{estimation}.
\begin{proof}[Proof of Lemma \ref{estimation}.]
Without loss of generality, we assume that $X$ starts from 0. Let $s$ be the scale function of $X$. It is easy to check that $s(X)$ is the solution to the stochastic differential equation
\begin{equation*}
dM_t = f(M_t)dW_t,\;\;\;Y_0 = 0.
\end{equation*}
where $f = (s'\sigma)\circ s^{-1}$.
Let $\sigma_z$ be the hitting time of $z$ by $s(X)$. Since $T =\sigma_{s(1)}\wedge\sigma_{s(-1)}$, the distribution of $T$ only depends on the value of $f$ on $[s(-1),s(1)]$, where $f^2$ is bounded from both below and above. Thus by Lemma \ref{prep}, there exists $\rho>0$ such that for any $t\geq 0$,
\begin{equation*}
P(T>t) = P(\sigma_{s(1)}\wedge\sigma_{s(-1)}>t) \geq e^{-\rho t}.
\end{equation*}
This completes the proof of this lemma.
\end{proof}

\section*{Acknowledgements}
The authors are grateful to J. Pitman, J.-F. Le Gall, Y. Liu, X.-F. Xue, R. Zhang, X. Chen, and the anonymous reviewers for their valuable comments and suggestions which greatly improved the quality of this paper. H. Ge is supported by NSFC (No. 10901040 and No. 21373021) and the Foundation for Excellent Ph.D. Dissertation from the Ministry of Education in China (No. 201119). C. Jia and D.-Q. Jiang are supported by NSFC (No. 11271029 and No. 11171024).

\setlength{\bibsep}{5pt}
\small\bibliographystyle{pnas2009}
\bibliography{diffusion}
\end{document}